\documentclass[11pt,letter]{article}
\usepackage[utf8]{inputenc}
\usepackage{amsmath}
\usepackage{amssymb}
\usepackage{amsthm}
\usepackage[english]{babel}
\usepackage{booktabs}
\usepackage{graphicx}
\usepackage[colorlinks,allcolors=blue]{hyperref}
\usepackage{pgfplots}
\pgfplotsset{compat=1.14}
\usepackage{xcolor}
\usepackage[margin=1in]{geometry}
\usepackage{caption}
\usepackage{subcaption}
\usepackage{nicefrac} 
\usepackage[affil-it]{authblk}
\usepackage{stmaryrd}
\usepackage{mathtools}
\usepackage[ruled,linesnumbered]{algorithm2e}
\usepackage{bm}
\usepackage{array}
\usepackage{empheq}

\usepackage{tabularx}

\usepackage[shortlabels]{enumitem}

\numberwithin{equation}{section}    

\usepackage{xpatch} 
\makeatletter
   \xpatchcmd{\@thm}{\fontseries\mddefault\upshape}{}{}{} 
\makeatother

\usepackage{float}

\usepackage{newtxtext}
\usepackage{newtxmath}

\newcommand{\innLtwo}[2]{\langle#1,#2 \rangle}
\newcommand{\innLtwolr}[2]{\left\langle#1,#2 \right\rangle}
\newcommand{\gegen}[2]{G_{#1}^{(#2)}}
\newcommand{\pol}[1]{\mathbb R_{#1}[x]}
\newcommand{\polC}[1]{\mathbb C_{#1}[x]}
\newcommand{\polquot}[2]{\mathbb R_{#1, #2}[x]}
\newcommand{\innprodQ}[2]{\langle #1, #2 \rangle_{\mathcal S}}

\newcommand{\innprodgegen}[2]{\langle #1, #2 \rangle_{\gamma}}

\definecolor{mygreen}{rgb}{0.0, 0.5, 0.0}

\newtheorem{theorem}{Theorem}[section]
\newtheorem{proposition}[theorem]{Proposition}
\newtheorem{lemma}[theorem]{Lemma}
\newtheorem{corollary}[theorem]{Corollary}

\theoremstyle{remark}
\newtheorem{remark}[theorem]{Remark}
\theoremstyle{definition}

\title{Characterization of the Tau Spectrum for Eigenvalue Problems}

\author[1,2]{Keaton J. Burns}
\author[3,4]{Ilaria Fontana\thanks{Corresponding author}}
\author[2,4,5]{Daniel Lecoanet}
\author[6]{Geoffrey M. Vasil}

\affil[1]{Department of Mathematics, Massachusetts Institute of Technology, Cambridge, MA, 02139}
\affil[2]{Center for Computational Astrophysics, Flatiron Institute, New York, NY, 10010}
\affil[3]{Department of Mathematics, University of Arizona, USA}
\affil[4]{Department of Engineering Sciences and Applied Mathematics, Northwestern University, Evanston, IL, 60208 USA}
\affil[5]{Center for Interdisciplinary Exploration and Research in Astrophysics, Northwestern University, Evanston, IL 60208 USA}
\affil[6]{School of Mathematics, University of Edinburgh, EH9 3FD, UK}

\date{}

\begin{document}

\maketitle

\begin{abstract}
    We consider tau approximations of a general eigenvalue problem for a $K$-th-order linear differential operator in one dimension, and we show that the approximate spectrum satisfies a resolvent-based spectral determinant property, which we use to investigate it. More specifically, we analyze the spectral determinant property for a second-order Sturm-Liouville eigenvalue problem approximated using Gegenbauer tau methods. Notably, a quantitative error analysis identifies two values of the Gegenbauer parameter~$\gamma$ leading to better approximation properties. These values correspond to Legendre polynomials and their derivatives. Moreover, we discuss the behavior of the largest eigenvalue of the tau approximation. The results are validated with a numerical implementation.
\end{abstract}


\medskip


\section{Introduction}

Spectral methods are widely used to numerically solve differential equations, providing fast convergence for problems with smooth solutions in simple geometries \cite{BOOK_Gottlieb_1977, BOOK_Canuto_1988, BOOK_Canuto_2006, BOOK_Boyd_2001}.
However, different implementations of spectral methods, primarily varying in how they incorporate boundary conditions, produce different approximate solutions.
In this paper, we are interested in analyzing finite truncations of the following eigenvalue problem: find $(\lambda, u) \in \mathbb C \times C^K([a,b])$ such that
\begin{equation}\label{eq:eigenvalue problem}
    \begin{cases}
        (\mathcal L - \lambda) u = 0 & \quad \text{in } \Omega\\
        \mathcal B(u) = 0 & \quad \text{on } \partial\Omega,
    \end{cases}
\end{equation}
where $\Omega = [a, b] \subset \mathbb R$ is a real interval, $C^K([a,b])$ is the space of functions with continuous derivatives up to order $K$ over the interval $[a,b]$, $\mathcal L$ is a $K$-th order constant-coefficient differential operator of the type
\begin{equation}\label{eq:linear operator}
    \mathcal L \coloneqq \sum_{k=0}^K \mathcal L_k \mathcal D^k,
    \qquad\qquad
    \mathcal L_k\in \mathbb R, k = 0, \dots, K,
\end{equation}
with $\mathcal L_K\neq 0$, $\mathcal D\coloneqq d/dx$, $\lambda$ is the spectral value (eigenvalue), $u = u(x)$ is the eigenfunction, and $\mathcal B(u)$ denotes a two-point boundary operator:
\begin{equation}\label{eq:boundary operator}
    \mathcal B(u) = \sum_{k=0}^{K-1} \left(\mathcal B_k^{(0)} \mathcal D^k u |_a + \mathcal B_k^{(1)} \mathcal D^k u |_b\right),
    \qquad
    \mathcal B_k^{(0)}, \mathcal B_k^{(1)} \in \mathbb R^K, k = 0, \dots, K-1,
\end{equation}
with ${\rm rank} [\, \begin{matrix} \mathcal B^{(0)} & \mathcal B^{(1)} \end{matrix}\,] = K$, where $\mathcal B^{(0)}$ and $\mathcal B^{(1)}$ are the $K\times K$ matrices with columns $\mathcal B_k^{(0)}$ and  $\mathcal B_k^{(1)}$, $k = 0, \dots, K-1$, respectively.
We denote by $\mathcal L_{\mathcal B}$ the operator $\mathcal L$ restricted to the domain
\begin{equation*}
    D(\mathcal L_{\mathcal B}) = \{u\in C^K([a,b])\ :\ \mathcal B(u) = 0\}.
\end{equation*}
The exact spectrum of this eigenvalue problem is denoted by $\sigma(\mathcal L_{\mathcal B})$.
As a particular example of this class of problems, when $\mathcal L = \mathcal D^2$, $\mathcal B_0^{(0)} = [\begin{matrix}
    1 & 0
\end{matrix}]$,
$\mathcal B_0^{(1)} = [\begin{matrix}
    0 & 1
\end{matrix}]$,
$\mathcal B_1^{(0)} = \mathcal B_1^{(1)} = [\begin{matrix}
    0 & 0
\end{matrix}]$, and $[a, b] = [-1, 1]$, we recover the simple second-order  Sturm–Liouville  problem: find $(\lambda, u) \in \mathbb C \times C^2([a,b])$ such that
\begin{equation}\label{eq:second-order eigenvalue problem}
    \begin{cases}
         \mathcal D^2 u = \lambda u &\quad x\in [-1, 1], \\
        u(\pm 1) = 0.
    \end{cases}
\end{equation}
One of the reasons why we are interested in this particular problem is that we can numerically validate the results since exact eigenvalue-eigenfunction pairs $(\lambda_j, \varphi_j)$ are known \cite{Charalambides_2008-2}:
\begin{equation}\label{eq:second-order exact}
    \lambda_{j} = - \left(\frac{j\pi}{2}\right)^2 
    \qquad\quad\text{and}\qquad\quad
    \varphi_{j}(x) = \begin{cases}
        \cos (\sqrt{\lvert\lambda_{j}\rvert} x) &\quad j \in \mathbb N, j \text{ odd},\\
        \sin(\sqrt{\lvert\lambda_{j}\rvert} x) &\quad j \in \mathbb N, j \text{ even}.
    \end{cases}
\end{equation}
As a consequence, denoting by $\mathcal D^2_{\rm Dir}$ the second-derivative operator restricted to the domain of twice-differentiable functions over $[-1, 1]$ satisfying the Dirichlet boundary conditions $u(\pm1) = 0$, the spectrum of \eqref{eq:second-order eigenvalue problem} is
\begin{equation*}
    \sigma\bigl(\mathcal D^2_{\rm Dir}\bigr) = \left\{- \left(\frac{j\pi}{2}\right)^2\right\}_{j\in \mathbb N}.
\end{equation*}

The \emph{tau method}, whose initial introduction goes back to \cite{Lanczos_1938}, is a class of spectral methods that consist of solving a perturbed equation where additional unknown coefficients, the $\tau$-terms, are added to the equation to simultaneously enforce the boundary conditions without the need to specialize the trial basis to the specific boundary conditions.
From now on, we denote by $\pol{k}$ and $\polC{k}$ the spaces of polynomials of degree at most $k$ with real and complex coefficients, respectively, and by $\polquot{n}{k}$ the quotient space $\pol{n}/\pol{n-k}$, for all integers $n, k\in\mathbb N_0$, $n\geq k$.
Consider a set of $K$ polynomials
\begin{equation}\label{eq:tau polynomials}
    \mathcal Q = \begin{bmatrix}
        Q_N & Q_{N-1} & \dots & Q_{N-K+1}
    \end{bmatrix}
\end{equation}
such that
\begin{equation}\label{eq:tau polynomials hypothesis}
    Q_{N-k} \in\pol{N-k}
    \qquad\forall k = 0, \dots, K-1,
    \qquad\text{and}\qquad
    {\rm span} (\mathcal Q) \simeq \polquot{N}{K}.
\end{equation}
Then, the $\tau$-perturbed problem of degree $N$ of the eigenvalue problem \eqref{eq:eigenvalue problem} reads: find $(\lambda, u, \tau) \in \mathbb C\times \polC{N} \times \mathbb C^K$,
such that
\begin{equation}\label{eq:tau perturbed problem}
    \begin{cases}
        (\mathcal L - \lambda) u = \mathcal Q\cdot \tau & \quad \text{in } \Omega\\
        \mathcal B(u) = 0 & \quad \text{on } \partial\Omega,
    \end{cases}
\end{equation}
where
\begin{equation*}
    \tau =
    \begin{bmatrix}
        \tau_1 \\
        \tau_2 \\
        \vdots \\
        \tau_K
    \end{bmatrix}
    \in\mathbb C^K
\end{equation*}
is the vector of (unknown) tau coefficients.
In classical tau formulations \cite{BOOK_Lanczos_1956, Clenshaw_1956, Coleman_1976, Chaves_1968}, the same test and trial functions are used and the last rows of the discrete matrix representing the linear operator are replaced by the boundary conditions.
This typically produces dense matrices and is equivalent to taking the $\tau$-polynomials $\mathcal Q$ as a subset of the test/trial functions.
Sparse spectral methods use different test and trial bases to render the interior operators banded.
Sparse schemes using recurrence and integral relations of orthogonal polynomials date back to \cite{Clenshaw_1956}.
Today, sparse methods are typically implemented in a Petrov-Galerkin fashion utilizing Chebyshev polynomials as trial bases and ultraspherical/Gegenbauer polynomials as test bases, and enforce boundary conditions via row replacement, equivalent to using the test family as $\tau$-polynomials \cite{Olver.2013}.

In this work, we consider the more general framework where $\mathcal Q$ can be any set of polynomials satisfying \eqref{eq:tau polynomials hypothesis}.
Additionally, the polynomial solutions $u\in\polC{N}$ of the $\tau$-perturbed problem do not depend on the choice of test and trial bases \cite{Burns_2026, Vasil_2019}, and the numerical implementation can be done efficiently \cite{Burns_2020}.
In particular, the $\tau$-perturbed problem for the second-order problem \eqref{eq:second-order eigenvalue problem} is defined as follows: given two polynomials $Q_N\in\mathbb R_N[x]$ and $Q_{N-1}\in\mathbb R_{N-1}[x]$ such that ${\rm span} (Q_N, Q_{N-1}) \simeq \polquot{N}{2}$, find $(\lambda, u, \tau) \in \mathbb C\times \polC{N} \times \mathbb C^2$, $\tau = [\tau_1, \tau_2]^T$, such that
\begin{equation}\label{eq:second-order tau}
    \begin{cases}
         \mathcal D^2 u = \lambda u + \tau_1 Q_N + \tau_2 Q_{N-1} &\quad x\in [-1, 1], \\
        u(\pm 1) = 0.
    \end{cases}
\end{equation}
The qualitative behavior of these approximations has been studied in \cite{Charalambides_2008-2} for Jacobi $\tau$-polynomials, i.e., orthogonal polynomials on the interval $[-1, 1]$ associated with the weight $(1-x)^{\alpha} (1+x)^{\beta}$.
In the ultraspherical case $\alpha = \beta = \gamma - 1/2$, the approximate eigenvalues are real, negative and distinct for $-\nicefrac{1}{2} <\gamma \leq \nicefrac{5}{2}$, while for $\gamma > \nicefrac{5}{2}$ pairs of complex-conjugate eigenvalues appear for sufficiently large $N$. 
For collocation discretizations, Weideman and Trefethen~\cite{Weideman_1988} showed that only a fraction of the eigenvalues are accurate, namely those whose eigenfunctions are resolved by the grid, and bounded the spectral radius via the characteristic polynomial.
However, to our knowledge, the asymptotic error of the resolved eigenvalues, and how it depends on the choice of the $\tau$-polynomials, has not been quantified.

The goal of this paper is twofold.
First, in Section~\ref{sec:spectral determinant}, we characterize the $\tau$-approximated spectrum with a resolvent-based spectral determinant property (Theorem~\ref{th:spectral determinat}), and we show that the spectral determinant factorizes into three factors (Theorem~\ref{th:factorization}). Notably, these results hold for any $K$-th order eigenvalue problem of the form \eqref{eq:eigenvalue problem}, and for any choice of the $\tau$-polynomial set \eqref{eq:tau polynomials}.
Second, in Section~\ref {sec:analysis}, we quantitatively study the eigenvalue approximation error for the second-order eigenvalue problem \eqref{eq:second-order eigenvalue problem}, looking at its asymptotic behavior as $N\to\infty$. In particular, we consider the case where the $\tau$-polynomials are Gegenbauer polynomials, i.e., orthogonal polynomials with respect to the weight $(1-x^2)^{\gamma-\nicefrac{1}{2}}$. The analysis of Subsection~\ref{subsec:error analysis} identifies two values of the Gegenbauer parameter $\gamma$ providing better eigenvalue approximations: $\gamma=\nicefrac{1}{2}$ and $\gamma=\nicefrac{3}{2}$, corresponding to Legendre polynomials and their derivatives (Theorem~\ref{th:asymptotic behavior N}, Theorem~\ref{th:main}, and Corollary~\ref{cor:rates}).
This characterizes the asymptotic behavior of the low, resolved part of the spectrum. At the opposite end, in Subsection~\ref{sec:largest}, we study the largest $\tau$-eigenvalues, whose magnitudes grow like $N^4$ and govern the stiffness of the discretization, with implications for the stability of time-stepping methods.
We also present a numerical validation of the theoretical results.

\section{Spectrum characterization: Spectral determinant}\label{sec:spectral determinant}

In the following, let $w_{\mathcal S} \colon [a, b] \to (0, \infty)$ be a weight function, $w_{\mathcal S}\in L^1(a,b)$, and let
\begin{equation*}
    \innprodQ{f}{g} \coloneqq \int_a^b f(x) g(x) w_{\mathcal S}(x) \,{\rm d}x,
\end{equation*}
for all functions $f, g$ belonging to the space
\begin{equation*}
    L_{w_{\mathcal S}}^2(a,b) \coloneqq \left\{ f\colon [a, b]\to\mathbb C\ :\ \int_a^b \lvert f(x)\rvert^2 w_{\mathcal S}(x) \,{\rm d}x < \infty \right\}.
\end{equation*}
The form $\innprodQ{\,\cdot\,}{\,\cdot\,}$ is a symmetric bilinear form on $L^2_{w_{\mathcal S}}(a,b)$, so that $\innprodQ{f}{g\,\tau}=\innprodQ{f}{g}\,\tau$ for $\tau\in\mathbb C$; on real-valued functions it is an inner product.
We will denote by $\{S_n\}_{n\in\mathbb N_0}$, $S_{n} \in\mathbb R_{n}[x]$, a sequence of orthogonal polynomials with respect to the inner product $\innprodQ{\,\cdot\,}{\,\cdot\,}$.
For simplicity of notation, given $f\in L^2_{w_{\mathcal S}}(a,b)$ and two row-vectors of functions
\begin{equation*}
    \mathcal F_1 = \begin{bmatrix}
        F_{1,1} & \dots & F_{1, m}
    \end{bmatrix}, \mathcal F_2 = \begin{bmatrix}
        F_{2,1} & \dots & F_{2, m}
    \end{bmatrix} \in (L^2_{w_{\mathcal S}}(a,b))^m,
\end{equation*}
we will still write $\innprodQ{\,\cdot\,}{\,\cdot\,}$ for the following vectors and matrix:
\begin{equation*}
    \innprodQ{\mathcal F_1^T}{f} = \begin{bmatrix}
        \innprodQ{F_{1,1}}{f} \\
        \vdots \\
        \innprodQ{F_{1,m}}{f}
    \end{bmatrix},
    \qquad\qquad
    \innprodQ{\mathcal F_{2}}{f} = \begin{bmatrix}
        \innprodQ{F_{2,1}}{f} & \dots & \innprodQ{F_{2,m}}{f}
    \end{bmatrix},
\end{equation*}
and
\begin{equation*}
    \innprodQ{\mathcal F_1^T}{\mathcal F_2}
    = \begin{bmatrix}
        \innprodQ{F_{1,1}}{F_{2,1}} & \dots & \innprodQ{F_{1,1}}{F_{2,m}} \\
        \vdots & \ddots & \vdots \\
        \innprodQ{F_{1,m}}{F_{2,1}} & \dots & \innprodQ{F_{1,m}}{F_{2,m}}
    \end{bmatrix}.
\end{equation*}

In the case $[a, b] = [-1, 1]$ and weight $w_{\gamma}(x) = (1-x^2)^{\gamma-\nicefrac{1}{2}}$, $\gamma > -1/2$, we recover the standard Gegenbauer polynomial family $\{\gegen{n}{\gamma}\}_{n\in\mathbb N_0}$, and we simply write $\innprodgegen{\,\cdot\,}{\,\cdot\,}$ to denote the corresponding inner product.
In the case $\gamma = \nicefrac{1}{2}$, the weight is trivial, $w(x)=1$, and we simply denote by $\innLtwo{\,\cdot\,}{\,\cdot\,}$ the corresponding inner product. The corresponding orthogonal polynomials are the Legendre polynomials $\{P_n\}_{n\in\mathbb N_0}$.

The following results characterize the discrete spectrum of the $\tau$-approximation \eqref{eq:tau perturbed problem}, by defining a \emph{spectral determinant} $D_N(\lambda)$ based on the exact resolvent with the physical boundary conditions, i.e., $(\mathcal L_{\mathcal B} - \lambda)^{-1}$.

\begin{theorem}[Spectral determinant]\label{th:spectral determinat}
    Given a linear operator $\mathcal L$ of order $K$ of the form \eqref{eq:linear operator} and a boundary operator $\mathcal B$ of the form \eqref{eq:boundary operator}, let $N \geq K$, let $\mathcal Q= [\begin{matrix}
        Q_N & Q_{N-1} & \dots & Q_{N-K+1}
    \end{matrix}]$ be a set of polynomials satisfying \eqref{eq:tau polynomials hypothesis} and $\{S_n\}_{n\in\mathbb N_0}$, $S_{n} \in\mathbb R_{n}[x]$, be a sequence of orthogonal polynomials with respect to the inner product $\innprodQ{\,\cdot\,}{\,\cdot\,}$, and
    $\mathcal S = [\begin{matrix}
        S_N & S_{N-1} & \dots & S_{N-K+1}
    \end{matrix}]$.
    If $(\lambda, u, \tau)$ is a solution of the $\tau$-perturbed problem \eqref{eq:tau perturbed problem}, such that $\lambda\notin \sigma(\mathcal L_{\mathcal B})$ and $u\in \polC{N}$, $u\neq 0$, 
    then
    \begin{equation}\label{eq:spectral determinant}
        D_N(\lambda) \coloneqq \det \left( \innprodQ{\mathcal S^T}{\mathcal D^K (\mathcal L_{\mathcal B} - \lambda)^{-1} \mathcal Q} \right) = 0.
    \end{equation}
\end{theorem}

\begin{proof}
    Notice that $\tau\neq 0$, otherwise $\lambda$ and $u$ would satisfy $(\mathcal L-\lambda) u = 0$ and $\mathcal B(u) = 0$, contradicting the hypothesis $\lambda\notin\sigma(\mathcal L_{\mathcal B})$.
    Moreover, since $\lambda$ is not an element of the spectrum of $\mathcal L_{\mathcal B}$, the operator $(\mathcal L_{\mathcal B} - \lambda)^{-1}$ is defined, and, consequently, the $\tau$-perturbed equation in \eqref{eq:tau perturbed problem} can be rewritten as
    \begin{equation*}
        u = (\mathcal{L}_{\mathcal B} - \lambda)^{-1} \mathcal{Q}\cdot \tau.
    \end{equation*}
    Since, by hypothesis, the solution $u$ is a polynomial of degree at most $N$, $\mathcal D^K u$ is a polynomial of degree at most $N-K$. Then, by orthogonality and linearity, we get
    \begin{equation*}
        0 = \innprodQ{\mathcal S^T}{\mathcal D^K u} =  \innprodQ{\mathcal S^T}{\mathcal D^K((\mathcal{L}_{\mathcal B} - \lambda)^{-1} \mathcal{Q} \cdot \tau)} = \innprodQ{\mathcal S^T}{\mathcal D^K (\mathcal L_{\mathcal B} - \lambda)^{-1} \mathcal Q}\cdot \tau.
    \end{equation*}
    Since $\tau\neq 0$, this is equivalent to the spectral determinant condition \eqref{eq:spectral determinant}.
\end{proof}

\begin{corollary}\label{cor:spectral determinant}
    Under the assumptions of Theorem~\ref{th:spectral determinat}, suppose the linear operator $\mathcal L_{\mathcal B}$ is symmetric, i.e., $\innLtwo{\mathcal L u}{v} = \innLtwo{u}{\mathcal L v}$ for all $u, v\in D(\mathcal L_{\mathcal B})$, and that $w_{\mathcal S}\in L^2(a,b)$. Then, $\mathcal L_{\mathcal B}$ has real eigenvalues $\{\lambda_j\}_{j\in\mathbb N}$, and corresponding real-valued eigenfunctions $\{\varphi_j\}_{j\in\mathbb N}$ forming an orthonormal basis of $L^2(a,b)$, i.e., $\innLtwo{\varphi_j}{\varphi_k} = \delta_{jk}$. Moreover, for $\lambda\notin\sigma(\mathcal L_{\mathcal B})$
    \begin{equation}\label{eq:rewriting determinant}
        \innprodQ{\mathcal S^T}{\mathcal D^K (\mathcal L_{\mathcal B} - \lambda)^{-1} \mathcal Q} = \sum_{j\in\mathbb N}   \frac{\innprodQ{\mathcal S^{T}}{\mathcal{D}^{K} \varphi_{j}} \innLtwo{\mathcal{Q}}{\varphi_{j}}}{\lambda_{j} - \lambda },
    \end{equation}
    and the spectral determinant property \eqref{eq:spectral determinant} can be rewritten as
    \begin{equation}
    \label{eq:determinant spectrum}
        \det \left( \sum_{j\in\mathbb N}   \frac{\innprodQ{\mathcal S^{T}}{\mathcal{D}^{K} \varphi_{j}} \innLtwo{\mathcal{Q}}{\varphi_{j}}}{\lambda_{j} - \lambda } \right) \ = \ 0.
    \end{equation}
\end{corollary}

\begin{proof}
    Since $\mathcal L_{\mathcal B}$ is symmetric, and $\mathcal L$ and $\mathcal B$ have real coefficients, \eqref{eq:eigenvalue problem} is a self-adjoint boundary value problem on a finite interval~\cite{BOOK_Coddington_1955, BOOK_Weidmann_1980}.
    Hence its eigenvalues are real, with $\lvert\lambda_j\rvert\to\infty$, its eigenfunctions can be chosen real-valued and form an orthonormal basis of $L^2(a,b)$, and, for every $\lambda\notin\sigma(\mathcal L_{\mathcal B})$, the kernel of $(\mathcal L_{\mathcal B}-\lambda)^{-1}$, i.e., the Green's function, is
    \begin{equation*}
        \mathcal K_{\lambda}(x,y)=\sum_{j\in\mathbb N}\frac{\varphi_j(x)\,\varphi_j(y)}{\lambda_j-\lambda}
        \qquad\text{in }L^2\bigl((a,b)^2\bigr).
    \end{equation*}
    Therefore, the following integral representation holds:
    \begin{equation}\label{eq:integral equation solution}
        (\mathcal L_{\mathcal B}-\lambda)^{-1}\mathcal Q
        =\int_a^b\mathcal K_{\lambda}(x,y)\,\mathcal Q(y)\,dy
        =\sum_{j\in\mathbb N}\frac{\varphi_j\,\innLtwo{\mathcal Q}{\varphi_j}}{\lambda_j-\lambda}
        \qquad\text{in }L^2(a,b).
    \end{equation}
    By the fact that $w_{\mathcal S}\in L^2(a,b)$ and the Cauchy--Schwarz inequality, each $\innprodQ{S_{N-k}}{\,\cdot\,}$,  $k = 0, \dots, K-1$, is continuous on $L^2(a,b)$, so $\innprodQ{\mathcal S^T}{\mathcal D^K\,\cdot\,}$ can be applied to~\eqref{eq:integral equation solution} term by term.
    This gives~\eqref{eq:rewriting determinant}, and combining it with~\eqref{eq:spectral determinant} gives~\eqref{eq:determinant spectrum}.
\end{proof}

In particular, for the second-order problem \eqref{eq:second-order eigenvalue problem}, the sets of eigenvalues and eigenvectors are given by \eqref{eq:second-order exact}, and the spectral determinant property is
\begin{equation}\label{eq:spectral determinant 2x2}
    \det \begin{bmatrix}
        \displaystyle\sum_{j\in\mathbb N} \frac{\innprodQ{S_N}{\mathcal D^2 \varphi_j} \innLtwo{Q_N}{\varphi_j}}{\lambda_j-\lambda} & \displaystyle\sum_{j\in\mathbb N} \frac{\innprodQ{S_N}{\mathcal D^2 \varphi_j} \innLtwo{Q_{N-1}}{\varphi_j}}{\lambda_j-\lambda} \smallskip\\
        \displaystyle\sum_{j\in\mathbb N} \frac{\innprodQ{S_{N-1}}{\mathcal D^2 \varphi_j} \innLtwo{Q_N}{\varphi_j}}{\lambda_j-\lambda} & \displaystyle\sum_{j\in\mathbb N} \frac{\innprodQ{S_{N-1}}{\mathcal D^2 \varphi_j} \innLtwo{Q_{N-1}}{\varphi_j}}{\lambda_j-\lambda}
    \end{bmatrix} = 0.
\end{equation}
Indeed, the set of eigenfunctions $\{\varphi_j\}_{j\in\mathbb N}$ defined by \eqref{eq:second-order exact} is a complete orthonormal basis for $L^2 ([-1, 1])$ equipped with $\innLtwo{\,\cdot\,}{\,\cdot\,}$.

\subsection{Factorization of the spectral determinant}

In this subsection, we present a second characterization of the spectral determinant that combines two main ingredients: a fundamental system of solutions of $(\mathcal L -\lambda) u = 0$, and an explicit inverse of $\mathcal L - \lambda$ on the polynomial space of degree at most $N$, given by a terminating Neumann series. In particular, we get a characterization of the $\tau$-eigenvalues corresponding to polynomial solutions $u$ that does not involve the set of orthogonal polynomials $\mathcal S$.

Given a linear operator $\mathcal L$ of order $K$ and $\lambda\in \mathbb C$, we denote by $\Phi_\lambda$ a fundamental system of solutions of $(\mathcal L-\lambda)\varphi = 0$, i.e., $K$ linearly independent solutions:
\begin{equation}\label{eq:fundamental solution basis}
    \Phi_{\lambda} = [\begin{array}{cccc}
        \varphi_{1,\lambda} & \varphi_{2, \lambda} & \dots & \varphi_{K, \lambda}
    \end{array}]
    \qquad\text{with}\qquad
    (\mathcal L-\lambda) \varphi_{k, \lambda} = 0,
    \qquad k = 1, \dots, K.
\end{equation}
In the following, we denote by $\mathcal B(\Phi_{\lambda})$ the following boundary condition matrix
\begin{equation*}
    \mathcal B(\Phi_{\lambda}) \coloneqq [\begin{array}{cccc}
        \mathcal B(\varphi_{1,\lambda}) & \mathcal B(\varphi_{2, \lambda}) & \dots & \mathcal B(\varphi_{K, \lambda})
    \end{array}].
\end{equation*}
The following lemma is a classical characterization of the spectrum $\sigma(\mathcal L_{\mathcal B})$ \cite{Birkhoff_1908, BOOK_Coddington_1955}.

\begin{lemma}\label{lem:characteristic determinant}
    Given a linear operator $\mathcal L$ of order $K$ of the form \eqref{eq:linear operator} together with a boundary operator $\mathcal B$ of the form \eqref{eq:boundary operator}, let $\lambda\in\mathbb C$ and let $\Phi_{\lambda}$ be the basis of fundamental solutions defined by \eqref{eq:fundamental solution basis}. Then, $\lambda \in\sigma(\mathcal L_{\mathcal B})$ if and only if $\det\mathcal B(\Phi_\lambda) = 0$.
\end{lemma}

\begin{proof}
    Note that nontrivial solutions of $(\mathcal L -\lambda) u = 0$ are $u = \Phi_{\lambda}\,a$ with $a\in\mathbb C^{K}$, $a\neq 0$. The coefficient vector $a$ is then constrained by the boundary conditions: $0 = \mathcal B(u) = \mathcal B(\Phi_{\lambda}) \cdot a$ by linearity.
    Then, the result follows directly from the fact that $\Phi_{\lambda}$ is a basis for the kernel of $\mathcal L - \lambda$. Indeed,
    \begin{equation*}
        \lambda \in \sigma(\mathcal L_{\mathcal B})
        \qquad\Leftrightarrow\qquad
        \exists\, a\neq 0 \text{ such that } \mathcal B(\Phi_{\lambda}) \cdot a = 0
        \qquad\Leftrightarrow\qquad
        \det\mathcal B(\Phi_{\lambda}) = 0.
    \end{equation*}
\end{proof}

We remark that the result of Lemma~\ref{lem:characteristic determinant} is independent of the choice of the specific basis $\Phi_{\lambda}$. More specifically, if $\Psi_{\lambda} = \Phi_{\lambda} A$, $\det A\neq 0$, is another basis, then by linearity $\det\mathcal B(\Psi_{\lambda}) = \det A \det \mathcal B(\Phi_{\lambda})$.

\begin{lemma}\label{lem:inverse G}
    Given a linear operator $\mathcal L$ of order $K$ of the form \eqref{eq:linear operator}, let $\lambda\neq \mathcal L_0$. Then, the operator $(\mathcal L-\lambda)|_{\polC{N}}$ is invertible, with inverse given by the following finite geometric sum:
    \begin{equation}\label{eq:inverse G}
        \mathcal G_{\lambda}^N \coloneqq - \sum_{n = 0}^N \frac{(\mathcal L - \mathcal L_0)^n}{(\lambda - \mathcal L_0)^{n+1}}.
    \end{equation}
\end{lemma}

\begin{proof}
    To prove this lemma, we show that $(\mathcal L-\lambda) \mathcal G_{\lambda}^N P = P$ for all polynomials $P\in\polC{N}$. Using the definition of $\mathcal G_{\lambda}^N$ \eqref{eq:inverse G} and the fact that $\mathcal L - \lambda = (\mathcal L - \mathcal L_0) - (\lambda - \mathcal L_0)$, we get
    \begin{equation*}
        (\mathcal L-\lambda) \mathcal G_{\lambda}^N P = \sum_{n=0}^N \frac{(\mathcal L - \mathcal L_0)^n}{(\lambda - \mathcal L_0)^{n}} P - \sum_{n=0}^N \frac{(\mathcal L - \mathcal L_0)^{n+1}}{(\lambda - \mathcal L_0)^{n+1}} P
        = P - \frac{(\mathcal L - \mathcal L_0)^{N+1}}{(\lambda - \mathcal L_0)^{N+1}} P = P.
    \end{equation*}
    The last equality follows from the fact that $(\mathcal L - \mathcal L_0)^{N+1} P = 0$ since $P\in\polC{N}$ and the linear operator $\mathcal L - \mathcal L_0$ does not contain any constant term.
\end{proof}

\begin{theorem}[Factorization]\label{th:factorization}
    Given a linear operator $\mathcal L$ of order $K$ of the form \eqref{eq:linear operator} and a boundary operator $\mathcal B$ of the form \eqref{eq:boundary operator}, let $N \geq K$, let $\mathcal Q= [\begin{matrix}
        Q_N & Q_{N-1} & \dots & Q_{N-K+1}
    \end{matrix}]$ be a set of polynomials satisfying \eqref{eq:tau polynomials hypothesis}, let $\lambda\notin\sigma(\mathcal L_{\mathcal B})$, $\lambda\neq \mathcal L_0$, let $\Phi_{\lambda}$ be the basis of fundamental solutions defined by \eqref{eq:fundamental solution basis} and let $\mathcal G_{\lambda}^N$ be 
    defined by \eqref{eq:inverse G}.
    \begin{enumerate}
        \item [(a)] Then, \begin{equation*}
        (\mathcal L_{\mathcal B} -\lambda)^{-1} \mathcal Q = \mathcal G_{\lambda}^N \mathcal Q - \Phi_{\lambda} (\mathcal B(\Phi_{\lambda}))^{-1} \mathcal B(\mathcal G_{\lambda}^N \mathcal Q).
    \end{equation*}
        \item [(b)] \label{it:factorization} If $\mathcal S$ and $\innprodQ{\,\cdot\,}{\,\cdot\,}$ are defined as in Theorem~\ref{th:spectral determinat}, the spectral determinant \eqref{eq:spectral determinant} factorizes as
        \begin{equation}\label{eq:factorization}
            D_N(\lambda)
            =(-1)^K\,
            \frac{A_N(\lambda) F_N(\lambda)}{E(\lambda)},
        \end{equation}
        where $A_N(\lambda) \coloneqq \det\innprodQ{\mathcal S^T}{\mathcal D^K\Phi_\lambda}$, $F_N(\lambda) \coloneqq \det\mathcal B(\mathcal G_\lambda^N\mathcal Q)$, and $E(\lambda) \coloneqq \det\mathcal B(\Phi_\lambda)$.
    \end{enumerate}
\end{theorem}

\begin{proof}
    Given $\tau\in\mathbb C^K$, $\tau\neq 0$, the general solution of $(\mathcal L - \lambda) u = \mathcal Q\cdot \tau$ is
    \begin{equation*}
        u = \Phi_{\lambda} \cdot a + \mathcal G_{\lambda}^N \mathcal Q \cdot \tau,
        \qquad a\in\mathbb C^K,
    \end{equation*}
    since $\mathcal G_{\lambda}^N \mathcal Q \cdot \tau$ is a particular solution by Lemma~\ref{lem:inverse G}.
    By applying the boundary conditions
    $\mathcal B(u) = 0$, by linearity we have
    \begin{equation*}
        \mathcal B(\Phi_{\lambda}) \cdot a + \mathcal B(\mathcal G_{\lambda}^N \mathcal Q) \cdot \tau = 0
        \qquad\Rightarrow\qquad
        a =  - \mathcal B(\Phi_{\lambda})^{-1} \mathcal B(\mathcal G_{\lambda}^N \mathcal Q) \cdot \tau.
    \end{equation*}
    Notice that $\mathcal B(\Phi_{\lambda})$ is invertible by Lemma~\ref{lem:characteristic determinant}.
    The solution with boundary conditions is thus
    \begin{equation*}
        u = (\mathcal L_{\mathcal B} - \lambda)^{-1} \mathcal Q \cdot \tau = \left(- \Phi_{\lambda} \mathcal B(\Phi_{\lambda})^{-1} \mathcal B(\mathcal G_{\lambda}^N \mathcal Q) + \mathcal G_{\lambda}^N \mathcal Q \right) \cdot \tau.
    \end{equation*}
    Since $\tau\in\mathbb C^K$, $\tau\neq 0$, is arbitrary, we get (a).
    Finally,
    \begin{equation*}
        \innprodQ{\mathcal S^T}{\mathcal D^K (\mathcal L_{\mathcal B} - \lambda)^{-1} \mathcal Q}
        = - \innprodQ{\mathcal S^T}{\mathcal D^K \Phi_{\lambda}} (\mathcal B(\Phi_{\lambda}))^{-1} \mathcal B(\mathcal G_{\lambda}^N \mathcal Q),
    \end{equation*}
    because $\mathcal D^K \mathcal G_{\lambda}^N \mathcal Q\in\pol{N-K}$. This leads to (b).
\end{proof}

\begin{corollary}\label{cor:characterization eigenvalues}
    Given a linear operator $\mathcal L$ of order $K$ and a boundary operator $\mathcal B$, let $\lambda\neq \mathcal L_0$. Then, $\lambda$ is an eigenvalue of the $\tau$-perturbed problem \eqref{eq:tau perturbed problem} with polynomial eigenfunction $u\neq 0$ if and only if
    \begin{equation*}
        F_N(\lambda) = \det\mathcal B(\mathcal G_{\lambda}^N \mathcal Q) = 0.
    \end{equation*}
    In that case, $u = \mathcal G_{\lambda}^N \mathcal Q\cdot \tau$ and $\tau\in\ker \mathcal B(\mathcal G_{\lambda}^N \mathcal Q)\setminus \{0\}$.
\end{corollary}

\begin{remark}[Poles and zeros of the spectral determinant]
    The factorization \eqref{eq:factorization} separates the spectral determinant $D_N(\lambda)$ into three parts:
    (1) the \emph{exact boundary determinant} $E(\lambda)$ vanishes exactly at the exact eigenvalues $\lambda\in\sigma(\mathcal L_{\mathcal B})$ (Lemma~\ref{lem:characteristic determinant}), which are therefore poles of $D_N(\lambda)$ unless cancelled by zeros of the numerator;
    (2) the \emph{finite tau determinant} $F_N(\lambda)$ vanishes at the numerical eigenvalues, i.e., values $\lambda\in \mathbb C$, $\lambda\neq \mathcal L_0$, such that there exists $u\in\polC{N}$, $u\neq 0$, and $\tau\in\mathbb C^K$ and $(\lambda, u, \tau)$ is a solution of the $\tau$-perturbed problem \eqref{eq:tau perturbed problem};
    (3) the \emph{auxiliary moment determinant} $A_N(\lambda)$ determines additional (spurious) zeros of $D_N(\lambda)$, which are not $\tau$-eigenvalues.
    This shows that the spectral determinant condition \eqref{eq:spectral determinant} is a necessary but not sufficient condition for characterizing numerical eigenvalues $\lambda$ and polynomial solution $u$.
\end{remark}

\begin{remark}[Dependence on $\mathcal S$]\label{rem:dependence S}
    While both the spectral and auxiliary determinants $D_N(\lambda)$ and $A_N(\lambda)$ depend on the choice of the test polynomials $\mathcal S$, the exact and finite tau ones $E(\lambda)$ and $F_N(\lambda)$ do not. This is expected since both exact and numerical eigenvalues are independent of the choice of $\mathcal S$ \cite{Burns_2026}. In particular, Corollary~\ref{cor:characterization eigenvalues} is the eigenvalue counterpart of the solvability criterion for $\tau$-perturbed boundary value problems in \cite[Lemma~2.3]{Burns_2026}.
\end{remark}

\section{Analysis of the spectral determinant}\label{sec:analysis}

Theorem~\ref{th:spectral determinat} and Theorem~\ref{th:factorization} hold for any differential eigenvalue problem \eqref{eq:eigenvalue problem}, any orthogonal family of test polynomials $\mathcal S$ and any set of $\tau$-polynomials $\mathcal Q$.
In this section, we analyze the spectral determinant in the following setting: we consider the second-order problem \eqref{eq:second-order eigenvalue problem} together with its $\tau$-perturbed version \eqref{eq:second-order tau} with $\tau$-polynomials from the family of Gegenbauer polynomials with parameter $\gamma$, i.e., $\mathcal Q = [\begin{matrix}
    \gegen{N}{\gamma} & \gegen{N-1}{\gamma}
\end{matrix}]$, where $\{\gegen{N}{\gamma}\}_{N\in\mathbb N_0}$ is the family of polynomials orthogonal with respect to the weight $w_{\gamma}=(1-x^2)^{\gamma-\nicefrac{1}{2}}$. 
By parity, this choice, together with $\mathcal S = \mathcal Q$ below, makes the $2\times 2$ matrix in \eqref{eq:spectral determinant 2x2} diagonal, while still covering a one-parameter family that includes Chebyshev and Legendre.
The main properties of Gegenbauer polynomials we will use in this paper are summarized in Appendix~\ref{app:properties Gegenbauer}. In particular, we use a nonstandard normalization for Gegenbauer polynomials in the spirit of \cite{Charalambides_2008-2, Charalambides_2008}.
Our goal is to analytically characterize how the quantitative approximation properties of the spectrum of the $\tau$-problem depend on the parameter $\gamma$.

\subsection{Spectral determinant and eigenvalue error analysis}\label{subsec:error analysis}

From now on, we take the test polynomials equal to the $\tau$-polynomials, with their own inner product:
\[
\mathcal S=\mathcal Q=\big[G_N^{(\gamma)}\ \ G_{N-1}^{(\gamma)}\big],
\qquad
\langle\,\cdot\,,\,\cdot\,\rangle_{\mathcal S} = \langle\,\cdot\,,\,\cdot\,\rangle_\gamma ,
\]
so that a single family of polynomials appears in the spectral determinant. Remark~\ref{rem:legendre tests} compares with Legendre test polynomials.
For the sake of brevity, from now on, we use the following notation for the parity conditions: $j\equiv N\,(2)$ if $j$ and $N$ have the same parity, $j\not\equiv N\,(2)$ otherwise.

\begin{proposition}[Second-order Gegenbauer spectral determinant]\label{prop:determinant}
Let $(\lambda,u,\tau)$ solve the $\tau$-perturbed problem \eqref{eq:second-order tau} with $\lambda\notin\sigma(\mathcal D^2_{\rm Dir})$ and $u$ a polynomial of degree at most $N$, and let $\{(\lambda_j, \varphi_j)\}_{j\in\mathbb N}$ be the set of eigenpairs defined by \eqref{eq:second-order exact} associated with the second-order problem \eqref{eq:second-order eigenvalue problem}. Then
\begin{equation}\label{eq:determinant}
\begin{aligned}
&\Bigg(\sum_{j\not\equiv N\,(2)}\frac{\lambda_j\,\langle G_N^{(\gamma)},\varphi_j\rangle_\gamma\langle G_N^{(\gamma)},\varphi_j\rangle}{\lambda_j-\lambda}\Bigg)\\
&\qquad\times
\Bigg(\sum_{j\equiv N\,(2)}\frac{\lambda_j\,\langle G_{N-1}^{(\gamma)},\varphi_j\rangle_\gamma\langle G_{N-1}^{(\gamma)},\varphi_j\rangle}{\lambda_j-\lambda}\Bigg)=0 .
\end{aligned}
\end{equation}
\end{proposition}

\begin{proof}
This result follows directly from Corollary~\ref{cor:spectral determinant} with $\mathcal S=\mathcal Q$ and $w_{\mathcal S}=w_\gamma$, together with $\langle G_n^{(\gamma)},\mathcal D^2\varphi_j\rangle_\gamma=\lambda_j\langle G_n^{(\gamma)},\varphi_j\rangle_\gamma$. In particular, by Gegenbauer parity \eqref{eq:parity}, the off-diagonal entries of the $2\times2$ matrix \eqref{eq:spectral determinant 2x2} vanish, so the determinant is the product of its diagonal entries.
\end{proof}

Throughout this section, we focus on the case $k\not\equiv N\, (2)$. The case $k\equiv N\, (2)$ is identical with $N-1$ in place of $N$.
Then, a $\tau$-eigenvalue close to $\lambda_k$ makes the first factor in \eqref{eq:determinant} vanish (the second factor is regular and non-zero near $\lambda_k$). Isolating the term $j=k$ gives an \emph{exact} formula for the error.

\begin{proposition}[Error identity]\label{prop:error identity}
Under the hypotheses of Proposition~\ref{prop:determinant}, if $\lambda$ is a $\tau$-eigenvalue for which the first factor in \eqref{eq:determinant} vanishes, then
\begin{equation}\label{eq:error identity}
\delta\lambda_k\coloneqq\lambda_k-\lambda
=-\,\lambda_k\,\frac{\langle G_N^{(\gamma)},\varphi_k\rangle_\gamma\,\langle G_N^{(\gamma)},\varphi_k\rangle}{\Sigma_{N,k}^{(\gamma)}(\lambda)},
\end{equation}
where
\begin{equation}\label{eq:reduced sum}
\Sigma_{N,k}^{(\gamma)}(\lambda)\coloneqq\sum_{\substack{j\not\equiv N\,(2)\\ j\neq k}}\frac{\lambda_j\,\langle G_N^{(\gamma)},\varphi_j\rangle_\gamma\langle G_N^{(\gamma)},\varphi_j\rangle}{\lambda_j-\lambda} .
\end{equation}
\end{proposition}

The numerator of \eqref{eq:error identity} is the product of the \emph{weighted} and the \emph{unweighted} Gegenbauer coefficients of the eigenfunction $\varphi_k$. The weighted one is given by a Bessel function, as shown by the following lemma.

\begin{lemma}[Weighted coefficients]\label{lem:bessel}
    Let $k\in\mathbb N$, $\gamma>-\nicefrac{1}{2}$, $N\in\mathbb N$ such that $N\not\equiv k (2)$, and $\kappa\coloneqq \nicefrac{k\pi}{2}$, so that $\lambda_k = -\kappa^2$. Then,
    \begin{align}
        \langle G_N^{(\gamma)},\varphi_k\rangle_\gamma
        &=(-1)^{\lfloor \nicefrac N2\rfloor}\,\frac{\pi\,\Gamma(N+2\gamma)\,J_{N+\gamma}(\kappa)}{2^{\gamma}\,N!\,\Gamma(\gamma+1)\,\kappa^{\gamma}},
        \label{eq:weighted coefficient}\\
        \|G_N^{(\gamma)}\|_\gamma^2
        &=\frac{\pi\,2^{-1-2\gamma}\,\Gamma(N+2\gamma)}{N!\,(N+\gamma)\,\Gamma(\gamma+1)^2},
        \label{eq:norm}\\
        \frac{\langle G_N^{(\gamma)},\varphi_k\rangle_\gamma}{\|G_N^{(\gamma)}\|_\gamma^2}
        &=(-1)^{\lfloor \nicefrac N2\rfloor}\,2^{\gamma+1}\Gamma(\gamma+1)\,(N+\gamma)\,\frac{J_{N+\gamma}(\kappa)}{\kappa^{\gamma}},
        \label{eq:ratio}
    \end{align}
    where $J_\nu$ denotes the Bessel function of the first kind.
    Moreover,
    \begin{equation*}
        0<J_{N+\gamma}(\kappa)<\frac{(\kappa/2)^{N+\gamma}}{\Gamma(N+\gamma+1)}
        \qquad \text{for } 0<\kappa<N+\gamma,
    \end{equation*}
    and
    \begin{equation}\label{eq:Bessel asymptotic}
        J_{N+\gamma}(\kappa)=\frac{(\kappa/2)^{N+\gamma}}{\Gamma(N+\gamma+1)}\bigl(1+O\bigl(\kappa^2 N^{-1}\bigr)\bigr)
        \qquad\text{as } N\to\infty.
    \end{equation}
    For the two special values of $\gamma$, the unweighted coefficient is a Bessel function as well:
    \begin{equation}\label{eq:unweighted special}
    \langle G_N^{(\nicefrac{1}{2})},\varphi_k\rangle=(-1)^{\lfloor \nicefrac N2\rfloor}\sqrt{\frac{2\pi}{\kappa}}\,J_{N+1/2}(\kappa),
    \qquad
    \langle G_N^{(\nicefrac{3}{2})},\varphi_k\rangle=(-1)^{\lfloor \nicefrac N2\rfloor}\,\frac13\sqrt{2\pi\kappa}\,J_{N+3/2}(\kappa).
    \end{equation}
\end{lemma}

\begin{proof}
The weighted coefficient \eqref{eq:weighted coefficient} is the Gegenbauer--Bessel integral \cite[Eq.~18.17.16\_5]{DLMF}
\begin{equation*}
    \int_{-1}^1(1-x^2)^{\gamma-1/2}C_N^{(\gamma)}(x)e^{i\kappa x}dx=2^{1-\gamma}\pi i^N\Gamma(N+2\gamma)\,\kappa^{-\gamma}\frac{J_{N+\gamma}(\kappa)}{N!\,\Gamma(\gamma)},
\end{equation*}
divided by $2\gamma$ (see Appendix~\ref{app:properties Gegenbauer}), taking its real part for $N$ even and its imaginary part for $N$ odd. The norm \eqref{eq:norm} is given in \cite[Table~18.3.1]{DLMF}, while the ratio \eqref{eq:ratio} is obtained by combining this norm with \eqref{eq:weighted coefficient}.
The lower bound for $J_{N+\gamma}$ follows from the fact that its first positive zero is greater than $N+\gamma$ \cite[Eq.~10.21.3]{DLMF}, while the upper bound follows from the infinite-product representation of $J_{N+\gamma}$ \cite[Eq.~10.21.15]{DLMF}. The asymptotic form for $J_{N+\gamma}$ follows from its series representation \cite[Eq.~10.2.2]{DLMF}.
Finally, the first identity in \eqref{eq:unweighted special} is \eqref{eq:weighted coefficient} with $\gamma=\nicefrac12$ (trivial weight). For the second one, we have $G_N^{(\nicefrac 32)}=\mathcal D P_{N+1}/3$ by \eqref{eq:derivative recurrence formula-bis} and it is sufficient to integrate by parts using $\varphi_k(\pm1)=0$, $\mathcal D\varphi_k=-\kappa\,\sin(\kappa x)$ for $N$ even and $\mathcal D\varphi_k=\kappa\,\cos(\kappa x)$ for $N$ odd, and apply \eqref{eq:weighted coefficient} with $N+1$.
\end{proof}

\begin{remark}[Why $\gamma=\nicefrac12$ and $\gamma=\nicefrac32$]\label{rem:why}
Since $$\langle G_N^{(\gamma)},\varphi_k\rangle=\langle G_N^{(\gamma)},(1-x^2)^{\nicefrac{1}{2}-\gamma}\varphi_k\rangle_\gamma,$$ the unweighted coefficient is the $\gamma$-Gegenbauer coefficient of $(1-x^2)^{\nicefrac{1}{2}-\gamma}\varphi_k$. Because $\varphi_k$ has simple zeros at $\pm1$, this function behaves like $(1\mp x)^{\nicefrac{3}{2}-\gamma}$ at the endpoints. It is therefore entire exactly when $\nicefrac32-\gamma$ is a non-negative integer, i.e., $\gamma\in\{\nicefrac{1}{2}, \nicefrac{3}{2}\}$; in that case its Gegenbauer coefficients decay super-algebraically \cite{Wang_2012, Wang_2016}. Otherwise, it has an algebraic endpoint singularity, whose effect is quantified by Theorem~\ref{th:asymptotic behavior N}.
\end{remark}

\begin{theorem}[Asymptotic behavior of $\langle \gegen{N}{\gamma}, \varphi_k\rangle$]\label{th:asymptotic behavior N}
    Given $k\in\mathbb N$, $\gamma > -\nicefrac{1}{2}$, and $N\in\mathbb N$ satisfying $N\not\equiv k\, (2)$, let $\varphi_k$ be the orthonormal basis function associated with the $k$-th eigenvalue $\lambda_k$ of \eqref{eq:second-order eigenvalue problem}, and let $\gegen{N}{\gamma}$ be the Gegenbauer polynomial of degree $N$ and parameter $\gamma$. Then
    \begin{equation}\label{eq:asymptotic inner product}
        \langle \gegen{N}{\gamma}, \varphi_k\rangle
        = C_{\gamma, k} N^{2\gamma - 5} \bigl(1+O\bigl(N^{-1}\bigr)\bigr),
        \qquad\text{as}\ N\to\infty
        \qquad\qquad
        \gamma \notin \{\nicefrac12, \nicefrac32\},
    \end{equation}
    where
    \begin{equation*}
        C_{\gamma, k} \coloneqq - \frac{2 \mathcal D \varphi_k(1)}{\Gamma(2\gamma+1)} (2\gamma-1) (2\gamma-3).
    \end{equation*}
    In particular, using the explicit exact solutions \eqref{eq:second-order exact},
    \begin{equation}\label{eq:constant with explicit}
        C_{\gamma, k} = (-1)^{\lceil\nicefrac{k}{2}\rceil + 1} \frac{k \pi}{\Gamma(2\gamma+1)} (2\gamma-1) (2\gamma-3).
    \end{equation}
    For the two excluded values $\gamma\in\{\nicefrac12,\nicefrac32\}$ the inner product
    decays faster than every power of $N$.
\end{theorem}

\begin{proof}
    Assume $\gamma\notin\{\nicefrac{1}{2}, \nicefrac{3}{2}\}$.
    Combining the properties of Gegenbauer polynomials \eqref{eq:parity} and \eqref{eq:derivative recurrence formula}, integration by parts, and the fact that
    $\varphi_k(\pm 1) = 0$ from \eqref{eq:second-order eigenvalue problem},
    we get
    \begin{equation*}
        \begin{split}
            \innLtwo{\gegen{N}{\gamma}}{\varphi_k} =&\ \innLtwolr{\frac{\mathcal D (\gegen{N+1}{\gamma} - \gegen{N-1}{\gamma})}{2(N+\gamma)}}{\varphi_k}
            = \underbrace{\left[ \frac{\gegen{N+1}{\gamma} - \gegen{N-1}{\gamma}}{2(N+\gamma)} \varphi_k \right]_{-1}^1}_{= 0} - \innLtwolr{\frac{\gegen{N+1}{\gamma} - \gegen{N-1}{\gamma}}{2(N+\gamma)}}{\mathcal D \varphi_k} \\
            =&\ - \innLtwolr{\frac{(N+\gamma-1) \mathcal D \gegen{N+2}{\gamma} - 2(N+\gamma) \mathcal D \gegen{N}{\gamma} + (N+\gamma+1) \mathcal D \gegen{N-2}{\gamma}}{4(N+\gamma-1) (N+\gamma) (N+\gamma+1)}}{\mathcal D \varphi_k}
        \end{split}
    \end{equation*}
    By defining
    \begin{equation}\label{eq:definition T_2,N,gamma}
        \mathfrak T_{2, N, \gamma} \coloneqq \frac{(N+\gamma-1) \gegen{N+2}{\gamma} - 2(N+\gamma) \gegen{N}{\gamma} + (N+\gamma+1) \gegen{N-2}{\gamma}}{4(N+\gamma-1) (N+\gamma) (N+\gamma+1)},
    \end{equation}
    and reapplying integration by parts and recalling that $(\lambda_k, \varphi_k)$ is an eigenpair of \eqref{eq:second-order eigenvalue problem}, we have
    \begin{equation}\label{eq:after two integrations}
        \begin{split}
            \innLtwo{\gegen{N}{\gamma}}{\varphi_k} = - \left[ \mathfrak T_{2, N, \gamma} \mathcal D \varphi_k \right]_{-1}^{1} + \innLtwo{\mathfrak T_{2, N, \gamma}}{\mathcal D^2 \varphi_k}
            \eqqcolon b_2^{(\gamma)}(N) + \lambda_k \innLtwo{\mathfrak T_{2, N, \gamma}}{\varphi_k}.
        \end{split}
    \end{equation}

    We focus on the first term in the last expression, using the parity of $\gegen{N}{\gamma}$ \eqref{eq:parity} and $\varphi_k$, and \eqref{eq:evaluation at 1} for the evaluation of the Gegenbauer polynomials at $x=1$:
    \begin{equation}\label{eq:b2 asymptotics}
        \begin{split}
            b_2^{(\gamma)}(N) =&\ - 2\mathfrak T_{2, N, \gamma}(1) \mathcal D\varphi_k(1) 
            = - 2 (2\gamma-1) (2\gamma-3) \, \frac{\Gamma(2\gamma+N-2)}{(N+2)! \Gamma(2\gamma+1)} \, \mathcal D\varphi_k(1)  \\
            =&\ C_{\gamma, k} N^{2\gamma - 5} \bigl(1+O\bigl(N^{-1}\bigr)\bigr),
        \end{split}
    \end{equation}
    where the last equality follows by using Stirling's approximation and the fact that $M! = \Gamma(M+1)$ for $M\in\mathbb{N}$ to show that $\Gamma(N+a)/\Gamma(N+b) = N^{a-b} (1+O(N^{-1}))$, $a,b\in\mathbb{R}$.

    For the sake of simplicity, we write $A\lesssim B$ when $A\leq C B$ for some $C>0$, and we use the notation $I_N\coloneqq \innLtwo{\gegen{N}{\gamma}}{\varphi_k}$.
    Using the Cauchy-Schwarz inequality, the fact that $\Vert\varphi_k\rVert = 1$, and the results in~\cite[Section 7.33]{BOOK_Szego_1975}, we get
    \begin{equation}\label{eq:a priori bound}
        \begin{split}
            \lvert I_N\rvert = \lvert\innLtwo{\gegen{N}{\gamma}}{\varphi_k}\rvert
            \leq \lVert \gegen{N}{\gamma} \rVert \, \lVert \varphi_k \rVert
            \leq \sqrt{2} \lVert \gegen{N}{\gamma} \rVert_{L^{\infty}} \lesssim N^q
        \end{split}
    \end{equation}
    for some finite $q = q(\gamma)$.
    Expanding $\mathfrak T_{2, N, \gamma}$ \eqref{eq:definition T_2,N,gamma} in \eqref{eq:after two integrations}, we get
    \begin{equation}\label{eq:linear combination}
        I_N = b_2^{(\gamma)}(N) + \lambda_k\left(\alpha_{N, 2} I_{N+2} + \alpha_{N, 0} I_N + \alpha_{N, -2} I_{N-2}\right)
    \end{equation}
    where $\alpha_{N,2}, \alpha_{N,0}, \alpha_{N,-2}$ are coefficients of order $N^{-2}$. As a consequence, if $\lvert I_N\rvert\lesssim N^{a}$ for $N\geq N_0$, then since $(N\pm2)^{a}\leq 2^{\lvert a\rvert} N^a$ for $N\geq 4$,
    \begin{equation*}
        \lvert I_N\rvert \lesssim \lvert b_2^{(\gamma)}(N)\rvert + N^{a-2}
        \qquad
        \text{for } N\geq \max\{N_0+2, 4\}.
    \end{equation*}
    By combining this bound with \eqref{eq:linear combination} iteratively, after $\lceil (q-2\gamma+5)/2\rceil$ steps we get $\lvert I_N\rvert \lesssim N^{2\gamma-5}$, and using this bound once more in combination with \eqref{eq:linear combination} we achieve
    \begin{equation*}
        \innLtwo{\gegen{N}{\gamma}}{\varphi_k} = b_2^{(\gamma)}(N) + O(N^{2\gamma-7}) = C_{\gamma, k} N^{2\gamma - 5} \bigl(1+O\bigl(N^{-1}\bigr)\bigr).
    \end{equation*}
    Finally, \eqref{eq:constant with explicit} follows from the fact that $\mathcal D \varphi_k (1) = (-1)^{\lceil\nicefrac{k}{2}\rceil} \frac{k\pi}{2}$.

    For the case $\gamma\in\{\nicefrac{1}{2}, \nicefrac{3}{2}\}$ see Remark~\ref{rem:why}.
\end{proof}

We remark that the leading asymptotic behavior follows from the properties of the second-order eigenvalue equation \eqref{eq:second-order eigenvalue problem}; the explicit eigenvalues and eigenfunctions $(\lambda_k, \varphi_k)$ \eqref{eq:second-order exact} are needed only to evaluate the constant $C_{\gamma, k}$ \eqref{eq:constant with explicit} which is independent of $N$.

Now, we give an explicit estimate of the eigenvalue error and its rate as $N\to \infty$.

\begin{theorem}[Eigenvalue error]\label{th:main}
    Let $k\in\mathbb N$, $\gamma>-\nicefrac{1}{2}$, $\kappa\coloneqq \nicefrac{k\pi}{2}$, so that $\lambda_k = -\kappa^2$, and
    \begin{equation*}
        D\coloneqq\{\lambda\in\mathbb C:\ \lvert\lambda-\lambda_k\rvert\leq r_k\},
        \qquad
        r_k\coloneqq\frac12\min\{\lvert\lambda_j-\lambda_k\rvert:\ j\neq k,\ j\equiv k\, (2)\}.
    \end{equation*}
    Then, there exists $N_0 = N_0(k, \gamma)$ such that, for every $N\geq N_0$ such that $N\not\equiv k\, (2)$
    \begin{enumerate}
        \item [(a)] the $\tau$-perturbed problem \eqref{eq:second-order tau} has exactly one eigenvalue $\lambda\in D$ with eigenfunction of the same parity as $\varphi_k$; this $\tau$-eigenvalue is real and different from $\lambda_k$;
        \item [(b)] at this eigenvalue, for the reduced sum \eqref{eq:reduced sum}, we have
        \begin{equation}\label{eq:reduce sum asymptotic}
            \Sigma_{N,k}^{(\gamma)}(\lambda)=\|G_N^{(\gamma)}\|_\gamma^2\Big(1+\frac{\kappa^2}{2(N+\gamma)^2-2}+O\big(\kappa^4N^{-4}\big)\Big);
        \end{equation}
        \item [(c)] for the error $\delta\lambda_k = \lambda_k-\lambda$, we have
        \begin{equation}\label{eq:main}
            \delta\lambda_k=\kappa^2\,\frac{\langle G_N^{(\gamma)},\varphi_k\rangle_\gamma\,\langle G_N^{(\gamma)},\varphi_k\rangle}{\|G_N^{(\gamma)}\|_\gamma^2}
            \,\left(1-\frac{\kappa^2}{2(N+\gamma)^2-2}+O\big(\kappa^4N^{-4}\big)\right).
        \end{equation}
        Equivalently,
        \begin{equation}\label{eq:main explicit}
            \delta\lambda_k=(-1)^{\lfloor N/2\rfloor}\,2^{\gamma+1}\Gamma(\gamma+1)\,(N+\gamma)\,\kappa^{2-\gamma}\,J_{N+\gamma}(\kappa)\,\langle G_N^{(\gamma)},\varphi_k\rangle\,\big(1+O(\kappa^2N^{-2})\big).
        \end{equation}
    \end{enumerate}
\end{theorem}

The proof is given in Appendix~\ref{app:proof}.
The main result of this theorem is \eqref{eq:main}, which expresses the error $\delta\lambda_k$, to leading order, as $\kappa^2=-\lambda_k$ times the product of two factors: (i) the $N$-th coefficient $\innprodgegen{\gegen{N}{\gamma}}{\varphi_k}/\lVert\gegen{N}{\gamma}\rVert^2_{\gamma}$ of $\varphi_k$ in its Gegenbauer expansion, and (ii) the unweighted Gegenbauer coefficient $\innLtwo{\gegen{N}{\gamma}}{\varphi_k}$.
The first factor is super-exponentially small for every $\gamma>-\nicefrac12$, see~\eqref{eq:ratio}. The second factor behaves differently if $\gamma\notin\{\nicefrac{1}{2}, \nicefrac{3}{2}\}$ or $\gamma\in\{\nicefrac{1}{2}, \nicefrac{3}{2}\}$. Namely, by Theorem~\ref{th:asymptotic behavior N}, $\delta\lambda_k$ is of order $N^{2\gamma-4} J_{N+\gamma}(\kappa)$ in the first case, and smaller in the second one.
This is stated by the following corollary, in which the explicit rates are shown.

\begin{corollary}[Rates]\label{cor:rates}
    Under the assumptions of Theorem~\ref{th:main}, as $N\to\infty$:
    \begin{enumerate}
        \item[(a)] for $\gamma\notin\{\nicefrac12,\nicefrac32\}$, with $C_{\gamma,k}$ from \eqref{eq:constant with explicit},
        \begin{equation}\label{eq:rate generic}
        \begin{split}
            \delta\lambda_k
            =&\ (-1)^{\lfloor \nicefrac{N}{2}\rfloor}\,2^{\gamma+1}\Gamma(\gamma+1)\,C_{\gamma,k}\;\kappa^{2-\gamma}\,N^{2\gamma-4}\,J_{N+\gamma}(\kappa)\,\big(1+O(N^{-1})\big), \\
            =&\ (-1)^{\lfloor \nicefrac{N}{2}\rfloor + \lceil\nicefrac{k}{2}\rceil+1} \frac{4 \Gamma(\gamma+1) (2\gamma-1)(2\gamma-3) \kappa^3}{\Gamma(2\gamma+1)}\, \frac{N^{2\gamma-4} (\kappa/2)^{N}}{\Gamma(N+\gamma+1)} \big(1+O(N^{-1})\big);
        \end{split}
        \end{equation}
        \item[(b)] for $\gamma\in\{\nicefrac12,\nicefrac32\}$,
        \begin{equation}\label{eq:rate special}
        \begin{split}
            \delta\lambda_k
            =&\ 2\pi\kappa\,(N+\gamma)\,(J_{N+\gamma}(\kappa))^2\,\big(1+O(\kappa^2N^{-2})\big)\\
            =&\ 2\pi\kappa\, \frac{(N+\gamma)\,(\kappa/2)^{2N+2\gamma}}{\Gamma(N+\gamma+1)^2}\big(1+O(\kappa^2N^{-1})\big).
        \end{split}
        \end{equation}
    \end{enumerate}
    Consequently,
    \begin{equation*}
        \begin{aligned}
            &\log_{10}|\delta\lambda_k|=-N\log_{10}\left(\frac{2N}{e\kappa}\right)+O(\log N)
            &\qquad\text{in case (a)}, \\
            &\log_{10}|\delta\lambda_k|=-2N\log_{10}\left(\frac{2N}{e\kappa}\right)+O(\log N)
            &\qquad\text{in case (b)},
        \end{aligned}
    \end{equation*}
    that is, the two special values of $\gamma$ asymptotically double the number of correct digits. In case (a) the sign of $\delta\lambda_k$ is that of $(-1)^{\lfloor \nicefrac{N}{2}\rfloor+\lceil \nicefrac{k}{2}\rceil+1}(2\gamma-1)(2\gamma-3)$. Figure~\ref{fig:low eigenvalues} shows these phenomena numerically.
\end{corollary}

\begin{proof}
Insert \eqref{eq:asymptotic inner product} into \eqref{eq:main explicit} for (a), and \eqref{eq:unweighted special} into $\eqref{eq:main explicit}$ together with the fact that $2^{\nicefrac32} \Gamma(\nicefrac{3}{2}) = \sqrt{2\pi}$ and $2^{\nicefrac52} \Gamma(\nicefrac{5}{2}) = 3 \sqrt{2\pi}$ for (b). Then, the second lines of \eqref{eq:rate generic} and \eqref{eq:rate special} follow from \eqref{eq:Bessel asymptotic} and \eqref{eq:constant with explicit}.
Finally, the logarithmic statements follow from Stirling's formula.
\end{proof}

\begin{remark}[Legendre test polynomials]\label{rem:legendre tests}
With $\mathcal S=[P_N\ P_{N-1}]$ and the unweighted inner product, the same proof gives
\[
\delta\lambda_k=\kappa^2\,\frac{\langle P_N,\varphi_k\rangle\,\langle G_N^{(\gamma)},\varphi_k\rangle}{\langle P_N,G_N^{(\gamma)}\rangle}\,\big(1+O(\kappa^2N^{-2})\big),
\]
\[
\langle P_N,\varphi_k\rangle=(-1)^{\lfloor N/2\rfloor}\sqrt{\frac{2\pi}{\kappa}}\,J_{N+1/2}(\kappa),
\qquad
\langle P_N,G_N^{(\gamma)}\rangle=\frac{\sqrt\pi\,\Gamma(N+\gamma)}{\Gamma(\gamma+1)\Gamma(N+\nicefrac12)(2N+1)} .
\]
Its leading term agrees with \eqref{eq:main explicit} to relative order $O(\kappa^2N^{-2})$: the ratio of the two is
\begin{equation*}
    \frac{J_{N+\nicefrac{1}{2}}(\kappa)\,\Gamma(N+\nicefrac32)}{J_{N+\gamma}(\kappa)\,\Gamma(N+\gamma+1)} \biggl(\frac{\kappa}{2}\biggr)^{\gamma-\nicefrac{1}{2}} = 1 + O(\kappa^2 N^{-2}).
\end{equation*}
This agrees with what we expect from the theory, since the $\tau$-spectrum does not depend on the test polynomials, see Remark~\ref{rem:dependence S}.
\end{remark}

\subsubsection{Numerical results}

To validate the previous theoretical results, we implemented and numerically solved the $\tau$-perturbed problem \eqref{eq:second-order tau} using 
Mathematica and the open-source software Dedalus \cite{Burns_2020},
with $(Q_N, Q_{N-1}) = (\gegen{N}{\gamma}, \gegen{N-1}{\gamma})$ for different values of $\gamma$. Since the eigenvalue errors quickly fall below double-precision round-off as $N$ increases, all the results reported in this section were computed with Mathematica. For $N=10$, the Dedalus implementation results agree with the Mathematica ones.
The $\tau$-perturbed problem is solved as the equivalent Petrov--Galerkin problem
\begin{equation*}
    \begin{cases}
        \langle G_m^{(\gamma)}, \mathcal D^2u-\lambda u\rangle_\gamma=0 &\quad  0\leq m\le N-2,\\
        u(\pm1)=0.
    \end{cases}
\end{equation*}
Figure~\ref{fig:low eigenvalues} shows the relative error for $N=64$ and the first four eigenvalues $\lambda_1 = -\nicefrac{\pi^2}{4}$, $\lambda_2 = -\pi^2$, $\lambda_3 = -\nicefrac{9\pi^2}{4}$, and $\lambda_4 = -4\pi^2$.
For all eigenvalues, the plots show two significant downward spikes corresponding to the two predicted special values $\gamma = \nicefrac{1}{2}$ and $\gamma = \nicefrac{3}{2}$, validating Corollary~\ref{cor:rates}.
Additionally, Table~\ref{tab:errors} reports the relative errors for the first eigenvalue for six representative values of $\gamma$ and for $N=10$ and $N=24$. At $N=10$, the special values $\gamma\in\{\nicefrac{1}{2}, \nicefrac{3}{2}\}$ yield relative errors of orders $10^{-15}$ and $10^{-18}$, compared with $10^{-10}$ and $10^{-11}$ for the other values of $\gamma$. The difference becomes more pronounced at $N=24$, where the errors decrease from $10^{-29}$ and $10^{-31}$ for the non-special values of $\gamma$ to $10^{-53}$ and $10^{-56}$ for $\gamma = \nicefrac{1}{2}$ and $\gamma = \nicefrac{3}{2}$, respectively. This behavior is consistent with the theoretical prediction of Corollary~\ref{cor:rates} that the values $\gamma\in\{\nicefrac{1}{2}, \nicefrac{3}{2}\}$ asymptotically double the number of correct decimal digits.

\begin{figure}[htb]
    \centering
    \includegraphics[width=0.6\textwidth]{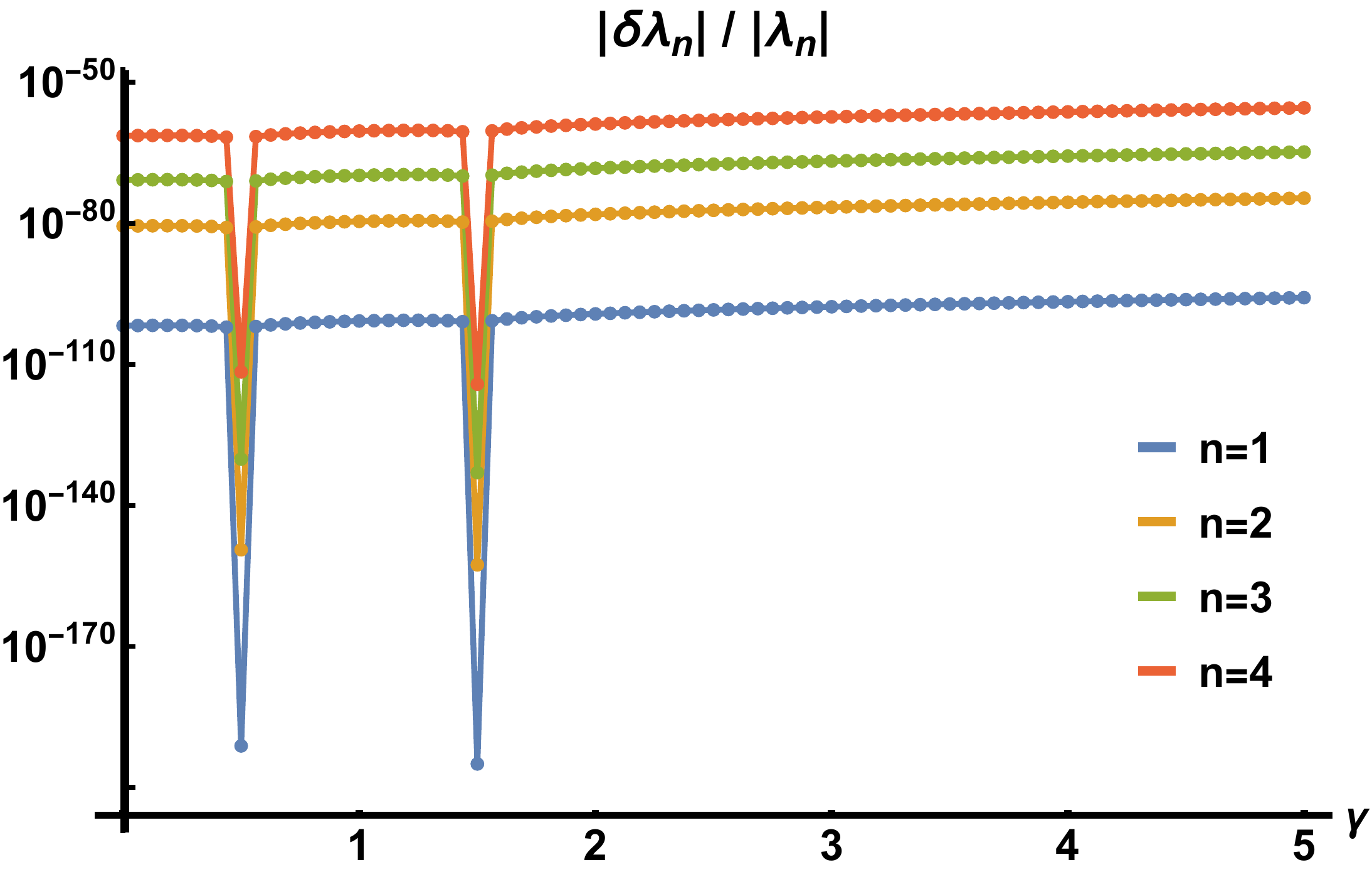}
    \caption{Relative error of the first four eigenvalues versus the Gegenbauer parameter, $0 \le \gamma \le 5$ with spacing $1/16$ and 1000-digit numerical precision. The exact eigenvalues are $\lambda_{n} = -(\nicefrac{n \pi}{2})^{2}$. The two spikes correspond to $\gamma = \nicefrac12, \nicefrac32$, Legendre $P_{N}(x)$ and a multiple of $\mathcal D P_{N+1}(x)$ respectively. The finite problem uses up to degree-$N=64$ modes. The $\tau$-corrections use $\gegen{64}{\gamma}$ and $\gegen{63}{\gamma}$ for each $\gamma$.}
    \label{fig:low eigenvalues}
\end{figure}

\begin{table}[htb]
    \centering
    \begin{tabular}{|c||cccccc|}
        \hline
        $N$ & 10 & 10 & 10 & 10 & 10 & 10 \\
        \hline
        $\gamma$ & $0$ & $\nicefrac{1}{2}$ & $1$ & $\nicefrac{3}{2}$ & $2$ & $\nicefrac{5}{2}$ \\
        \hline
        $\lvert\delta\lambda_{1}\rvert/\lvert\lambda_{1}\rvert$
        & $4.51\times10^{-11}$
        & $1.65\times10^{-15}$
        & $5.49\times10^{-11}$
        & $8.52\times10^{-18}$
        & $2.77\times10^{-10}$
        & $8.61\times10^{-10}$ \\
        \hline
        \hline
        $N$ & 24 & 24 & 24 & 24 & 24 & 24 \\
        \hline
        $\gamma$ & $0$ & $\nicefrac{1}{2}$ & $1$ & $\nicefrac{3}{2}$ & $2$ & $\nicefrac{5}{2}$ \\
        \hline
        $\lvert\delta\lambda_{1}\rvert/\lvert\lambda_{1}\rvert$
        & $2.73\times10^{-31}$
        & $7.07\times10^{-53}$
        & $9.60\times10^{-31}$
        & $7.00\times10^{-56}$
        & $1.15\times10^{-29}$
        & $5.24\times10^{-29}$ \\
        \hline
    \end{tabular}
    \caption{Relative error of the first eigenvalue for $N=10,\, 24$ and six representative values of $\gamma$.}
    \label{tab:errors}
\end{table}

\subsection{The largest eigenvalues}\label{sec:largest}

In this subsection, we briefly describe the $\tau$-eigenvalues of largest modulus in the parity class of $G_N^{(\gamma)}$ (for the other class, replace $N$ by $N-1$).
For this parity block, set $z=1/\lambda$, $\lambda\neq 0$. Then we consider the polynomial
\begin{equation*}
    p_{N}(z) \coloneqq \sum_{m=0}^{\lfloor N/2\rfloor}z^{m} \mathcal D^{2m} G_{N}^{(\gamma)}(1) \, = \, \sum_{m=0}^{\lfloor N/2\rfloor} (4z)^{m} 
    \frac{\Gamma(\gamma+2m+1)}{\Gamma(\gamma+1)}\, \frac{\Gamma(2\gamma + 2m + N)}{(N-2m)! \Gamma(2\gamma+4m+1)},
\end{equation*}
see \eqref{eq:mth evaluation at 1} for the evaluation of the derivative at $1$. $p_N(z)$ is the characteristic polynomial, and it can be derived by using $u=\tau_1\mathcal G_\lambda^N \gegen{N}{\gamma}$ by~\eqref{eq:inverse G} together with the boundary condition $u(1)=0$. Therefore, $\lambda$ is a $\tau$-eigenvalue in this parity class if and only if $p_N(\nicefrac{1}{\lambda}) = 0$.
In particular, the largest $\tau$-eigenvalue $\lambda$ corresponds to the root with smallest modulus of $p_{N}(z) = 0$.
This polynomial formulation is already in Charalambides--Waleffe~\cite[Section~3]{Charalambides_2008-2}. In that work, the authors showed using this polynomial that if $-\nicefrac{1}{2} <\gamma \leq \nicefrac{5}{2}$, then the $\tau$-eigenvalues of \eqref{eq:second-order tau} are real, negative and distinct, and this result is sharp, since if $\gamma>\nicefrac{5}{2}$ the $\tau$ spectrum contains pairs of conjugate complex eigenvalues. 

The polynomial $p_{N}(z)/G_{N}^{(\gamma)}(1) = \sum_{m} c_{m} z^{m} $ is of hypergeometric type with the ratio of coefficients,
\begin{align*}
    \frac{c_{m+1}}{c_{m}} \, &= \,  \frac{(2 m-N) (2 m-N+1) (2 \gamma +2 m+N) (2 \gamma +2
   m+N+1)}{(2 \gamma +4 m+1) (2 \gamma +4 m+3)}  \\
   & \, \sim \, \frac{(N+\gamma)^{4}}{(2 \gamma +4 m+1) (2 \gamma +4 m+3)} 
   \qquad \text{as} \ N\, \to \, \infty.
\end{align*}
Then it holds that
\begin{equation*}
    \lambda \sim -\frac{(N+\gamma)^{4}}{4\,\zeta_{\gamma}^{2}}  \qquad \text{as} \qquad N\,\to \, \infty,
\end{equation*}
where $\zeta_{\gamma}$ is the zero with smallest modulus of the Mittag--Leffler function \cite[Section~10.46]{DLMF},
\begin{equation*}
    \lim_{N \to \infty} \frac{p_{N}(-4\zeta^2/N^{4})}{\gegen{N}{\gamma}(1)}\frac{1}{\Gamma(\gamma+\nicefrac{1}{2})} =  W_{\gamma}(\zeta) \coloneqq \, E_{2,\gamma + 1/2}(-\zeta^{2}) \, = \,
    \sum_{m \ge 0} \frac{(-\zeta^{2})^{m}}{\Gamma(2m + \gamma + \nicefrac{1}{2})}.
\end{equation*}
Special cases include
\begin{equation*}
    W_{1/2}(\zeta) \, = \, \cos(\zeta), \qquad
   W_{3/2}(\zeta) \, = \, \frac{\sin(\zeta)}{\zeta}, \qquad
    W_{5/2}(\zeta) \, = \, \frac{2\sin^{2}(\zeta/2)}{\zeta^{2}},
\end{equation*}
so $\zeta_\gamma=\nicefrac\pi2,\pi,2\pi$, the last being a double root.
Using the following integral representation \cite{Haubold_2011} for $\gamma > \nicefrac{1}{2}$,
\begin{equation*}
    W_{\gamma}(\zeta) = \frac{1}{\Gamma(\gamma-\nicefrac{1}{2})} \int_0^1 (1-s)^{\gamma-\nicefrac{3}{2}} \cos(\zeta s)\, ds
\end{equation*}
and integrating by parts twice, for $\gamma > \nicefrac{5}{2}$ we get
\begin{equation*}
    W_{\gamma}(\zeta) = \frac{(\gamma-\nicefrac{3}{2})(\gamma-\nicefrac{5}{2})}{\Gamma(\gamma-\nicefrac{1}{2}) \zeta^2} \int_0^1 (1-s)^{\gamma-\nicefrac{7}{2}} (1-\cos(\zeta s))\, ds,
\end{equation*}
which is strictly positive for all $\zeta \in \mathbb R\setminus\{0\}$. Consequently, for $N$ sufficiently large, the largest $\tau$-eigenvalue is complex. As $\gamma\to \nicefrac{5}{2}^-$, the two smallest zeros of $W_{\gamma}$ merge at $2\pi$. This phenomenon can be observed in Figure~\ref{fig:largest} which shows $\lvert{\rm Re}(\lambda)\rvert$ and $\lvert{\rm Im}(\lambda)\rvert$ of the two largest eigenvalues in one parity block versus the parameter $0\leq \gamma\leq 5$ for the $\tau$-perturbed problem \eqref{eq:second-order tau} with $N=64$. The absolute real part of the largest eigenvalue decreases with $\gamma$. Additionally, both eigenvalues are real for $0\leq \gamma\leq \nicefrac{5}{2}$ and complex for $\nicefrac{5}{2} < \gamma \leq 5$, consistent with the result of \cite[Section 3.3]{Charalambides_2008-2}.

\begin{figure}[htb]
    \centering
    \includegraphics[width=0.6\textwidth]{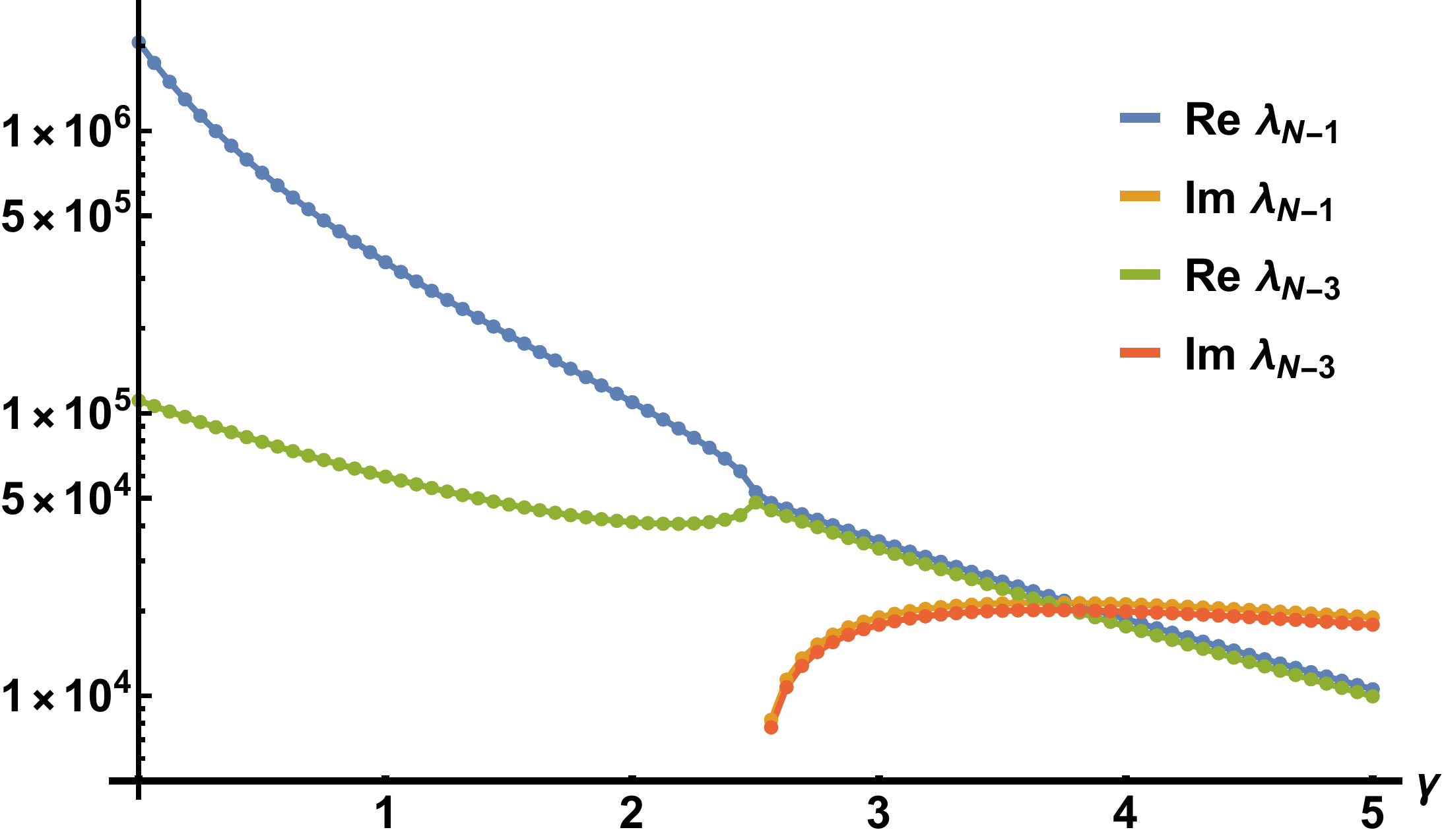}
    \caption{The absolute magnitudes near the spectrum cutoff. For $N=64$, the $\tau$-approximated spectrum contains $N+1$ polynomial modes, two $\tau$ parameters, and four formally $\infty$-eigenvalues due to the rank-deficient mass matrix. Hence, there are $N-1$ finite eigenvalues. We plot the largest one (index $N-1$) and the next one of the same parity (index $N-3$). Non-real eigenvalues appear only for $\gamma > 5/2$. For $N-1$, the ratio $\lambda{(\gamma=0)} / \lambda{(\gamma=5/2)} \approx 39.1$. As $N\to \infty$, the ratio approaches $\approx 47.8$.}
    \label{fig:largest}
\end{figure}

\section{Conclusions}

Given the eigenvalue problem for a constant-coefficient linear differential operator $\mathcal L$ of order $K$ on an interval together with appropriate boundary conditions, we have characterized its $\tau$-approximated spectrum with a spectral determinant property $D_N(\lambda) = 0$ (Theorem~\ref{th:spectral determinat}). This property can be analyzed with two approaches: either rewriting it using an orthonormal eigenfunction basis of $L^2(a,b)$ (Corollary~\ref{cor:spectral determinant}), or by factorizing it using a finite geometric sum (Theorem~\ref{th:factorization}). The factorization separates the exact eigenvalues, which are the zeros of the exact boundary determinant $E$ and hence generically poles of $D_N$, from the $\tau$-eigenvalues $\lambda\neq\mathcal L_0$, which are exactly the zeros of the finite tau determinant $F_N$ (Corollary~\ref{cor:characterization eigenvalues}), and from the zeros of the auxiliary moment determinant $A_N$, which depend on the choice of test polynomials.
We have applied these approaches to a second-order Sturm-Liouville problem with Gegenbauer $\tau$-polynomials, and obtained an explicit leading-order formula for the eigenvalue error (Theorem~\ref{th:main}). It shows that the choices $\gamma=\nicefrac{1}{2}$ and $\gamma=\nicefrac{3}{2}$ asymptotically double the number of correct digits (Corollary~\ref{cor:rates}). The reason is that only for these values do the unweighted Gegenbauer coefficients of the eigenfunctions decay faster than any power of $N$ (Remark~\ref{rem:why}). Additionally, the largest eigenvalue grows like $N^4$ with $\gamma$-dependent constant. It stays real for $\gamma\leq\nicefrac{5}{2}$ and becomes complex for larger $\gamma$ (Section~\ref{sec:largest}). Since this eigenvalue limits the stable time step of explicit schemes for time-dependent problems, $\gamma=\nicefrac{3}{2}$ is preferable: its constant is four times smaller than that of $\gamma=\nicefrac{1}{2}$.
Overall, this work provides an asymptotic description of both ends of the $\tau$-spectrum, each obtained from one of the two characterizations of Section~\ref{sec:spectral determinant}: the eigenfunction expansion of the spectral determinant gives the error of the low, resolved eigenvalues, while the finite tau determinant $F_N$ gives the growth and the real or complex nature of the largest eigenvalues. This combination can serve as a starting point for a similar analysis of more complex problems for which the optimal choice of the $\tau$-polynomials is not clear.
Natural extensions include variable-coefficient operators, other boundary conditions, and higher-order problems.

\appendix

\section{Properties of Gegenbauer polynomials}\label{app:properties Gegenbauer}

Given $\gamma > -1/2$, the {\it Gegenbauer polynomials} (or {\it ultraspherical polynomials}) $\{C_{N}^{(\gamma)}\}_{N\in\mathbb N_0}$ are the family of orthogonal polynomials on the interval $[-1,1]$ with respect to the weight $w_{\gamma}(x) = (1-x^2)^{\gamma-1/2}$.
Following \cite{Charalambides_2008}, we define the Gegenbauer polynomial family $\{G_{N}^{(\gamma)}\}_{N\in\mathbb N_0}$ with the following nonstandard normalization to avoid the singularity at $\gamma = 0$:
\begin{equation*}
    \begin{cases}
        \gegen{0}{\gamma}(x) = 1 \\
        \gegen{N}{\gamma}(x) = \displaystyle\frac{C_N^{(\gamma)}(x)}{2\gamma} & \qquad\text{if}\ N\geq 1.
    \end{cases}
\end{equation*}
Up to a normalization, they correspond to Jacobi polynomials $\{P_{N}^{(\alpha, \beta)}\}_{N\in\mathbb N_0}$ with $\alpha = \beta = \gamma-1/2$, and they include Legendre polynomials ($\gamma = 1/2$) and Chebyshev polynomials ($\gamma = 0$).

We summarize here some of the properties of Gegenbauer polynomials. We refer to~\cite{Charalambides_2008} and the references therein for their proofs.
\begin{itemize}
    \item {\bf Parity:} For all $N\in \mathbb N_0$ and $\gamma > -\nicefrac{1}{2}$,
    \begin{equation}\label{eq:parity}
        \gegen{N}{\gamma}(x) = (-1)^N \gegen{N}{\gamma}(-x)
        \qquad\forall x\in\mathbb R
    \end{equation}

    \item {\bf Derivative recurrence formulas:} For all $\gamma > -\nicefrac{1}{2}$,
    \begin{equation}\label{eq:derivative recurrence formula-bis}
        \begin{cases}
            \mathcal D \gegen{1}{\gamma}(x) = \gegen{0}{\gamma} = 1 & \\
            \mathcal D \gegen{N}{\gamma}(x) = 2(\gamma+1) \gegen{N-1}{\gamma+1} &\qquad N\geq 2
        \end{cases}
    \end{equation}
    Additionally,
    \begin{equation}\label{eq:derivative recurrence formula}
        \begin{cases}
            \gegen{0}{\gamma} (x) = \mathcal D \gegen{1}{\gamma} (x) & \smallskip\\
            \gegen{N}{\gamma} (x) = \displaystyle\frac{\mathcal D[\gegen{N+1}{\gamma} (x) - \gegen{N-1}{\gamma} (x)]}{2(N+\gamma)}  &\qquad N\geq 1
        \end{cases}
    \end{equation}

    \item {\bf $m$-th derivative formula:} For all $\gamma > -\nicefrac{1}{2}$,
    \begin{subequations}\label{eq:mth derivative formula}
        \begin{empheq}[left=\empheqlbrace]{align}
            &\mathcal D^m \gegen{N}{\gamma} (x) = 2^m \displaystyle\frac{\Gamma(\gamma+m+1)}{\Gamma(\gamma+1)} \gegen{N-m}{\gamma+m}(x)  &  m\leq N-1 \label{eq:mth derivative formula 1} \medskip\\
            &\mathcal D^{N} \gegen{N}{\gamma} (x) = 2^{N-1} \displaystyle\frac{\Gamma(\gamma+N)}{\Gamma(\gamma+1)} & \label{eq:mth derivative formula 2}
        \end{empheq}
    \end{subequations}

    \item {\bf Evaluation at $x=1$:} For all $N\geq 1$ and $\gamma > - \nicefrac{1}{2}$,
    \begin{equation}\label{eq:evaluation at 1}
        \gegen{N}{\gamma}(1) = \frac{\Gamma(2\gamma+N)}{N! \Gamma(2\gamma+1)}
    \end{equation}

    \item {\bf Evaluation at $x=1$ of the $m$-th derivative:}
    \begin{equation}\label{eq:mth evaluation at 1}
        \mathcal D^{m} \gegen{N}{\gamma}(1) = 2^{m} \frac{\Gamma(\gamma+m+1)}{\Gamma(\gamma+1)}\, \frac{\Gamma(2\gamma + m + N)}{(N-m)! \Gamma(2\gamma+2m+1)}
        \qquad m \leq N
    \end{equation}
    Note that, for $m\leq N-1$, this follows directly from the combination of \eqref{eq:mth derivative formula 1} and \eqref{eq:evaluation at 1}, while for $m = N$, this formula reduces to the right-hand side of \eqref{eq:mth derivative formula 2}.

\end{itemize}

\section{Proof of Theorem~\ref{th:main}}\label{app:proof}

\paragraph{Notation.}
For $\lambda\in\mathbb C$, write $\lambda=-\mu^2$ with $\mu\in\mathbb C$, let $\varphi_\mu(x)=\cos(\mu x)$ if $N$ is even, $\varphi_\mu(x)=\sin(\mu x)$ if $N$ is odd, so that $\varphi_\kappa=\varphi_k$, and let
\begin{equation*}
    c_N(\mu)\coloneqq-\frac{\mathcal D\varphi_\mu(1)}{2\varphi_\mu(1)}
    \qquad\Rightarrow\qquad
    c_N(\mu) = \frac{\mu}{2} \tan \mu \text{ if $N$ is even,}
    \qquad
    c_N(\mu) = -\frac{\mu}{2} \cot \mu \text{ if $N$ is odd.}
\end{equation*}
Throughout this proof, we simplify the notation as follows: $G=G_N^{(\gamma)}$, $w=w_\gamma$.

\paragraph{Step 1: The first factor in \eqref{eq:determinant} in closed form.}
For $\lambda\notin\sigma(\mathcal D^2_{\rm Dir})$, we want to show that
\begin{equation}\label{eq:secular}
    \Sigma_N^{(\gamma)}(\lambda) \coloneqq \sum_{j\not\equiv N\,(2)}\frac{\lambda_j\langle G,\varphi_j\rangle_\gamma\langle G,\varphi_j\rangle}{\lambda_j-\lambda}
    =\|G\|_\gamma^2-\frac{\mu}{2}\,I(\mu) + c_N(\mu)\,\langle G,\varphi_\mu\rangle_\gamma\langle G,\varphi_\mu\rangle,
\end{equation}
where
\begin{equation*}
    I(\mu)\coloneqq\iint_{[-1,1]^2} w(x)G(x)\,\sin(\mu|x-y|)\,G(y)\,dx\,dy.
\end{equation*}
With this goal, let $u=(\mathcal D^2-\lambda)^{-1}G$, i.e., $\mathcal D^2u+\mu^2u=G$, $u(\pm1)=0$.
Using $\lambda_j (\lambda_j-\lambda)^{-1}=1+\lambda(\lambda_j-\lambda)^{-1}$, we split the series $\Sigma_N^{(\gamma)}(\lambda)$ into two parts. Remarking that $\innLtwo{G}{\varphi_j} = 0$ when $j\equiv N\,(2)$ by parity,
the eigenfunction expansion $u=\sum_{j\in\mathbb N}\langle G,\varphi_j\rangle\varphi_j/(\lambda_j-\lambda)$ has coefficients of order $j^{-3}$, hence converges uniformly on $[-1,1]$; since $wG\in L^1(-1,1)$, the second part equals $\lambda\langle G,u\rangle_\gamma$.
The partial sums of $\sum_{j\in\mathbb N}\langle G,\varphi_j\rangle\varphi_j$ are Fourier partial sums of a piecewise smooth periodic extension of $G$, hence uniformly bounded and convergent to $G$ on $(-1,1)$~\cite{BOOK_Zygmund_2003}; by dominated convergence, the first part equals $\|G\|_\gamma^2$.
Therefore,  $\Sigma_N^{(\gamma)}(\lambda) =\|G\|_\gamma^2-\mu^2\langle G,u\rangle_\gamma$.
The Dirichlet Green's function of $\mathcal D^2+\mu^2$ on $[-1,1]$ is
\[
\mathcal K_\mu(x,y)=\frac{\sin(\mu|x-y|)}{2\mu}+\frac{\cos(2\mu)\cos(\mu(x-y))-\cos(\mu(x+y))}{2\mu\sin(2\mu)} ,
\]
and inserting $u=\int_{-1}^1\mathcal K_\mu(\cdot,y)G(y)dy$, the first term gives $-\frac{\mu}{2} I(\mu)$, and the second term, by parity, $c_N(\mu)\,\langle G,\varphi_\mu\rangle_\gamma\langle G,\varphi_\mu\rangle$, which proves \eqref{eq:secular}.

\paragraph{Step 2: Expansion of $I(\mu)$.}
Iterating the identity
\begin{equation}\label{eq:gegenbauer identity}
    \mathcal D\bigl[(1-x^2)^{\nu+1/2}G_{n-1}^{(\nu+1)}\bigr]
    =-\frac{n(n+2\nu)}{2(\nu+1)}(1-x^2)^{\nu-1/2}G_n^{(\nu)},
\end{equation}
see~\cite[\S18.9]{DLMF}, we get, for $0\leq l<N$,
\begin{equation*}
    wG=(-1)^lc_l\mathcal D^l
    \Bigl[(1-x^2)^{\gamma+l-1/2}G_{N-l}^{(\gamma+l)}\Bigr],
    \qquad
    c_l=\prod_{i=0}^{l-1}
    \frac{2(\gamma+i+1)}{(N-i)(N+2\gamma+i)}.
\end{equation*}
The bracket and its first $l-1$ derivatives vanish at $\pm1$.
Thus, using
$\mathcal D_x^{2m+2}|x-y|^{2m+1}
=2(2m+1)!\delta(x-y)$,
integration by parts gives, for $l=2m+2<N$,
\begin{equation}\label{eq:double integral}
    \frac{1}{(2m+1)!}
    \iint_{[-1,1]^2}
    w(x)G(x)|x-y|^{2m+1}G(y)\,dx\,dy
    =2c_l e_l\|G\|_\gamma^2,
\end{equation}
where
\begin{equation*}
    e_l=
    \frac{\langle(1-x^2)^lG_{N-l}^{(\gamma+l)},G\rangle_\gamma}
    {\|G\|_\gamma^2}.
\end{equation*}
For $l=2$, applying the Gegenbauer identity \eqref{eq:gegenbauer identity} twice and orthogonality,
\begin{equation*}
    e_2=
    -\frac{N(N-1)(N+2\gamma)(N+2\gamma+1)}
    {8(\gamma+1)(\gamma+2)((N+\gamma)^2-1)},
    \qquad\text{which implies}\qquad
    2c_2 e_2 = -\frac{1}{(N+\gamma)^2-1}.
\end{equation*}
Consequently,
\begin{equation}\label{eq:double integral paired}
    \iint_{[-1,1]^2}
    w(x)G(x)|x-y|G(y)\,dx\,dy
    =-\frac{\|G\|_\gamma^2}{(N+\gamma)^2-1}.
\end{equation}
For $l\le N/2$, Cauchy--Schwarz and $(1-x^2)^{2l}\le(1-x^2)^l$ imply
\begin{equation*}
    |e_l| \le\frac{\|G_{N-l}^{(\gamma+l)}\|_{\gamma+l}}{\|G\|_\gamma}.
\end{equation*}
Using the Gegenbauer norm formula \eqref{eq:norm}, we obtain
\begin{equation*}
    |c_l e_l|
    \le
    \left[(N+2\gamma)_l\frac{N!}{(N-l)!}\right]^{-1/2}
    \le\left(\frac2N\right)^l,
\end{equation*}
where $(a)_l=a(a+1)\cdots(a+l-1)$ is the Pochhammer symbol.
Expanding the sine kernel in its Taylor series, the term of order $\mu^{2m+2}$ in $\frac{\mu}{2}I(\mu)$ is $(-1)^m\mu^{2m+2}c_{2m+2}e_{2m+2}\|G\|_\gamma^2$ whenever $2m+2\le N/2$. By~\eqref{eq:double integral paired}, the leading term is
$-\mu^2\|G\|_\gamma^2/[2((N+\gamma)^2-1)]$. For $4|\mu|\le N$, the remaining terms with $2m+2\le N/2$ are bounded by
\begin{equation*}
    \|G\|_\gamma^2
    \sum_{\substack{l\ge4\\l\text{ even}}}
    \left(\frac{2|\mu|}{N}\right)^l
    \le
    2\left(\frac{2|\mu|}{N}\right)^4
    \|G\|_\gamma^2.
\end{equation*}
The higher-order terms with $2m+2>\nicefrac{N}{2}$ are bounded by
\begin{equation*}
    \|wG\|_{L^1}\|G\|_{L^1}
    \sum_{2m+2>N/2}
    \frac{(2|\mu|)^{2m+2}}{(2m+1)!}.
\end{equation*}
For bounded $|\mu|$, this tail is super-algebraically small relative to $\|G\|_\gamma^2$, since $\|wG\|_{L^1}\|G\|_{L^1}/\|G\|_\gamma^2$ grows at most polynomially in $N$ \cite{BOOK_Szego_1975}. Therefore, for any fixed $\mu_0>0$, uniformly for $|\mu|\le\mu_0$, and $N$ sufficiently large, we have
\begin{equation}\label{eq:I expansion}
    \frac{\mu}{2}I(\mu)
    =-\frac{\mu^2\|G\|_\gamma^2}
    {2((N+\gamma)^2-1)}
    +O(|\mu|^4N^{-4})\|G\|_\gamma^2.
\end{equation}

\paragraph{Step 3: The reduced sum on $D$.}
From~\eqref{eq:secular}, we get
\begin{equation*}
    \Sigma_{N,k}^{(\gamma)}(\lambda)
    =\|G\|_\gamma^2-\frac{\mu}{2}I(\mu) + R(\lambda),
    \qquad
    R(\lambda)\coloneqq c_N(\mu)\,\langle G,\varphi_\mu\rangle_\gamma\langle G,\varphi_\mu\rangle - \frac{\lambda_k\langle G,\varphi_k\rangle_\gamma\langle G,\varphi_k\rangle}{\lambda_k-\lambda}.
\end{equation*}
Since $D$ contains no $\lambda_j$ with $j\neq k$ and $j\equiv k\, (2)$, the reduced sum $\Sigma_{N,k}^{(\gamma)}(\lambda)$ is analytic on $D$, and so are $\frac\mu2I(\mu)$ and hence $R(\lambda)$.
Since $r_k\leq\nicefrac12\lvert\lambda_{k+2}-\lambda_k\rvert=(k+1)\pi^2/2$, we have $\lvert\mu\rvert\leq\kappa+\pi$ on $D$.
By~\eqref{eq:weighted coefficient} with $\kappa$ replaced by $\mu$, the bound $\lvert J_\nu(z)\rvert\leq\lvert z/2\rvert^\nu e^{\lvert{\rm Im}\,z\rvert}/\Gamma(\nu+1)$~\cite[Eq.~10.14.4]{DLMF}, and $\lvert\langle G,\varphi_\mu\rangle\rvert\leq e^{\lvert{\rm Im}\,\mu\rvert}\|G\|_{L^1}$, we obtain
\begin{equation*}
    \lvert c_N(\mu)\,\langle G,\varphi_\mu\rangle_\gamma\langle G,\varphi_\mu\rangle\rvert + \lvert\lambda_k\langle G,\varphi_k\rangle_\gamma\langle G,\varphi_k\rangle\rvert \leq
    C\,N^{q}\left(\frac{\kappa+\pi}{2}\right)^N \frac{1}{\Gamma(N+\gamma+1)},
    \qquad\lambda \in \partial D
\end{equation*}
with $C$ and $q$ depending only on $k$ and $\gamma$.
Moreover, on $\partial D$, $\lvert\lambda-\lambda_k\rvert=r_k$ and $c_N(\mu)$ is uniformly bounded, since the poles of $c_N$ are the $\lambda_j$ with $j\equiv k\, (2)$, all at distance at least $r_k$ from $\partial D$.
Hence, $R(\lambda)$ is super-exponentially small on $\partial D$ and, by the maximum modulus principle, on $D$. In particular, $R(\lambda) =O (N^{-4})\|G\|_\gamma^2$ on $D$.
Combining this with~\eqref{eq:I expansion} for $\mu^2=-\lambda$, we have, uniformly for $\lambda\in D$,
\begin{equation}\label{eq:reduced sum on D}
    \Sigma_{N,k}^{(\gamma)}(\lambda) = \|G\|_\gamma^2\Bigl(1-\frac{\lambda}{2(N+\gamma)^2-2}+O\bigl(\lvert\lambda\rvert^2 N^{-4}\bigr)\Bigr).
\end{equation}

\paragraph{Step 4: Existence and uniqueness.}
Restricted to the parity class of $\varphi_k$, \eqref{eq:second-order tau} has the single $\tau$-polynomial $G$ and the single boundary condition $u(1)=0$.
Note that $\lambda=0$ is not a $\tau$-eigenvalue, since $\mathcal D^2$ maps $\mathbb C_N[x]$ into $\mathbb C_{N-2}[x]$, making $\tau_1 =0$ and consequently $u=0$.
For $\lambda\neq 0$, since $(\mathcal D^2_{{\rm Dir}}-\lambda)^{-1}$ and $\mathcal G_\lambda^N$ \eqref{eq:inverse G} preserve parity, we may take the fundamental system $\Phi_\lambda=[\varphi_\mu,\psi_\mu]$, with $\psi_\mu$ of the parity opposite to $\varphi_\mu$, and the equivalent boundary conditions $\nicefrac12(u(1)\pm (-1)^N u(-1))=0$. Then, all matrices in the identity
\begin{equation*}
    \langle\mathcal S^T,\mathcal D^2(\mathcal D^2_{{\rm Dir}}-\lambda)^{-1}\mathcal Q\rangle_\gamma=-\langle\mathcal S^T,\mathcal D^2\Phi_\lambda\rangle_\gamma\,\mathcal B(\Phi_\lambda)^{-1}\mathcal B(\mathcal G_\lambda^N\mathcal Q)
\end{equation*}
of Theorem~\ref{th:factorization} are diagonal.
Their first diagonal entries give
\begin{equation*}
    \Sigma_N^{(\gamma)}(\lambda)=\mu^2\langle G,\varphi_\mu\rangle_\gamma\,\frac{(\mathcal G_\lambda^NG)(1)}{\varphi_\mu(1)}
\end{equation*}
for $\lambda\notin\sigma(\mathcal D^2_{\rm Dir})\cup\{0\}$, and, since both sides are analytic near the $\lambda_j$ with $j\equiv N\, (2)$, for all $\lambda\in D\setminus\{0,\lambda_k\}$.
By Corollary~\ref{cor:characterization eigenvalues}, $\lambda$ is a $\tau$-eigenvalue in this parity class if and only if $(\mathcal G_\lambda^NG)(1)=0$.
Moreover, since $\lambda =0$ is not a $\tau$-eigenvalue and $\innprodgegen{G}{\varphi_{\mu}}\neq 0$ by \eqref{eq:weighted coefficient} with $\mu$ instead of $\kappa$, $\lambda\in D\setminus\{0, \lambda_k\}$ is a $\tau$-eigenvalue in this parity class if and only if $\Sigma_N^{(\gamma)}(\lambda)=0$.
Let
\begin{equation*}
    f(\lambda)\coloneqq(\lambda_k-\lambda) \Sigma_N^{(\gamma)}(\lambda) = \lambda_k \langle G,\varphi_k\rangle_\gamma\langle G,\varphi_k\rangle + g(\lambda)
    \qquad\text{where}\qquad
    g(\lambda)\coloneqq(\lambda_k-\lambda)\Sigma_{N,k}^{(\gamma)}(\lambda).
\end{equation*}
Since $\varphi_{\mu}(1)$, as a function of $\lambda$, has a simple zero at $\lambda_k$, both $f(\lambda)$ and $g(\lambda)$ are analytic in $D$.
Moreover, $\lambda_k$ is not a $\tau$-eigenvalue and, by~\eqref{eq:reduced sum on D}, $g(\lambda)$ has exactly one zero in $D$ at $\lambda_k$. This implies that $\lambda\neq 0$ is a $\tau$-eigenvalue in $D$ if and only if $f(\lambda) = 0$.
Additionally, there exists $N_0$ such that, for $N\geq N_0$,
\begin{equation*}
    \lvert f(\lambda)-g(\lambda)\rvert
    = \lvert\lambda_k \langle G,\varphi_k\rangle_\gamma\langle G,\varphi_k\rangle \rvert
    < \frac12r_k\|G\|_\gamma^2\leq\lvert g(\lambda)\rvert,
    \qquad
    \lambda\in\partial D.
\end{equation*}
By Rouch\'e's theorem, $f(\lambda)$ has exactly one zero $\lambda\in D$, counted with multiplicity.
Since $\Sigma_N^{(\gamma)}(\lambda)$ has real coefficients and $\overline{\lambda}\in D$, $f(\overline\lambda)=\overline{f(\lambda)}$, and $\overline\lambda$ is also a zero of $f$ in $D$. By uniqueness, this implies that $\lambda$ is real. This proves (a).

\paragraph{Step 5: Asymptotics.}
Let now $\lambda$ be the unique $\tau$-eigenvalue in $D$ from Step 4. Since $\Sigma_N^{(\gamma)}(\lambda) = 0$, we have $\lambda_k \langle G,\varphi_k\rangle_\gamma\langle G,\varphi_k\rangle = -(\lambda_k-\lambda)\Sigma_{N,k}^{(\gamma)}(\lambda)$, and therefore
\begin{equation*}
    \lvert\delta\lambda_k\rvert=\frac{\lvert\lambda_k \langle G,\varphi_k\rangle_\gamma\langle G,\varphi_k\rangle  \rvert}{\lvert\Sigma_{N,k}^{(\gamma)}(\lambda)\rvert}\leq\frac{2\lvert\lambda_k \langle G,\varphi_k\rangle_\gamma\langle G,\varphi_k\rangle  \rvert}{\|G\|_\gamma^2},
\end{equation*}
so $\delta\lambda_k = \lambda_k-\lambda$ is super-exponentially small. Combining this with \eqref{eq:reduced sum on D}, we get \eqref{eq:reduce sum asymptotic}. Then, \eqref{eq:error identity} implies \eqref{eq:main}. Finally, \eqref{eq:main explicit} follows by inserting \eqref{eq:ratio} into \eqref{eq:main}.

\bibliographystyle{siam}
\bibliography{analysis_tau_refs}

\end{document}